\documentclass[11pt,reqno]{amsart}

\usepackage{amscd,amsmath,amssymb,amsfonts,verbatim,graphicx,enumitem}
\usepackage[colorlinks]{hyperref}
\usepackage{slashed}

\newcommand{\fm}{\mathfrak{m}}

\renewcommand{\epsilon}{\varepsilon}
\renewcommand{\leq}{\leqslant}
\renewcommand{\geq}{\geqslant}

\def\XXint#1#2#3{{\setbox0=\hbox{$#1{#2#3}{\int}$ }
\vcenter{\hbox{$#2#3$ }}\kern-.6\wd0}}

\newtheorem{theorem}{Theorem}[section]

\newtheorem{lemma}[theorem]{Lemma}
\newtheorem{corollary}[theorem]{Corollary}

\newtheorem{assumption}[theorem]{Assumption}

\theoremstyle{definition}
\newtheorem{definition}[theorem]{Definition}
\newtheorem{remark}[theorem]{Remark}

\numberwithin{equation}{section}

\begin{document}

\title{Asymptotics of high-codimensional area-minimizing currents in hyperbolic space
}

\author{Xumin Jiang}
\address{School of Sciences, Great Bay University, Dongguan 523000, China}
\email{xjiang@gbu.edu.cn}

\author{Jiongduo Xie}
\address{School of Mathematical Sciences\\
Shanghai Jiao Tong University\\
Shanghai, 200240, China}
\email{jiongduoxie@outlook.com
}
	
\date{\today}

\begin{abstract}
    We investigate the asymptotic behavior of high-codimensional area-minimizing locally rectifiable currents in hyperbolic space, addressing a problem posed by F.H. Lin \cite{Lin1989CPAM} and establishing     ``boundary regularity at infinity" results for such currents near their asymptotic boundaries under the standard Euclidean metric. Intrinsic obstructions to high-order regularity arise for odd-dimensional minimal surfaces, revealing a constraint dependent on the geometry of the asymptotic boundary. Our work advances the asymptotic theory of high-codimensional minimal surfaces in hyperbolic space.
\end{abstract}

\maketitle

\markboth{High codimensional area-minimizing currents in hyperbolic space}{ Xumin Jiang and Jiongduo Xie}

\setcounter{tocdepth}{3}
\tableofcontents

\section{Introduction}
In the upper half-plane model, the $\fm$-dimensional hyperbolic space is given by the set
\begin{align}
    \mathbb{H}^\fm = \bigl\{ (x', x^\fm) \in \mathbb{R}^{\fm-1} \times \mathbb{R}_+ \bigr\}, \qquad \fm \geq 2.
\end{align}
Set $\fm = n+k$, where $n\geq 2$ and $k\geq 1$ are natural numbers. Let $\Gamma$ be a closed $C^{1,\alpha}$ submanifold of dimension $n-1$ in $\mathbb{R}^{\fm-1}\times\{0\}$ for some constant $\alpha\in(0,1)$. In \cite{Anderson1982}, M.\ Anderson proved that there exists an area-minimizing, locally rectifiable $n$-current $T$, which is complete, without boundary, and asymptotic to $\Gamma$ at infinity. See also \cite{Anderson1983}.

Hardt and Lin \cite{Hardt&Lin1987} investigated the hypersurface case ($k=1$) and established the ``boundary regularity at infinity'' result: for any such hyperbolic-area-minimizing current $T$, the union of the support of $T$ (denoted $M$) and $\Gamma$—when endowed with the Euclidean metric—is a finite union of $C^{1,\alpha}$ hypersurfaces with boundary $\Gamma$ in a neighborhood of $\Gamma$. These hypersurfaces intersect $\mathbb{R}^n \times \{0\}$ orthogonally along $\Gamma$, and all interior singularities of $M$ are confined to a bounded region of $\mathbb{H}^\fm$. 
Lin \cite{Lin1989Invent} further established a higher-order boundary regularity result: if $\Gamma$ is a $C^{l,\alpha}$ submanifold for $l = 2, 3, \dots, n$, then $M \cup \Gamma$ is a union of some $C^{l,\alpha}$ smooth hypersurfaces with boundary $\Gamma$ near $\Gamma$ (in the Euclidean metric). See also \cite{Tonegawa1996MathZ} and \cite{Lin2012Invent}.

If $\Gamma$ is a $C^{l,\alpha}$ submanifold for some integer $l \geq n+1$, then there exists a geometric obstruction that prevents each $C^{n,\alpha}$ hypersurface in  $M\cup\Gamma$ from being a $C^{l,\alpha}$ hypersurface with boundary $\Gamma$ in a neighborhood of $\Gamma$. As a prototypical example, when $n=3$, each $C^{n,\alpha}$ hypersurface in $M\cup\Gamma$ is a $C^\infty$ hypersurface with boundary $\Gamma$ near $\Gamma$ if and only if $\Gamma$ is a Willmore surface. By definition, this requires the mean curvature $H$ and the Gaussian curvature $K$ of $\Gamma$ to satisfy the Willmore equation
\begin{align}
    \Delta H + 2H(H^2 - K)=0.
\end{align}
Han and Jiang \cite{HanJiang2023} analyzed such geometric obstructions to high-order regularity and established several associated regularity results. 
In a subsequent work, Han and Jiang \cite{HanJiang1} proved a convergence theorem under the additional assumption that $\Gamma$ is analytic. For related results, we refer the reader to \cite{FHJ1, HanJiang2024, HanXie, JiangShi, JiangXiao}. An important application of such precise asymptotic behavior lies in the gluing program developed by Fu, Hein, and Jiang \cite{FHJ2}. See also \cite{JSZ}.

\medskip
In this paper, we focus on the high codimension setting, namely \(k\geq 2\). Federer \cite{Federer} constructed explicit examples showing that the interior of minimal surfaces in this regime admits codimension-two singularities. A canonical illustration is provided by complex submanifolds in \(\mathbb{C}^{2}\): these objects are all locally absolutely area-minimizing minimal surfaces, whose singularities consist precisely of real codimension-two branching points.
 By the classical interior regularity theory established by F. Almgren \cite{Almgren}, the support of any hyperbolic-area-minimizing locally rectifiable $n$-current is a relatively closed subset of $\mathbb{H}^\fm$, and forms a real analytic submanifold off a relatively closed singular set whose Hausdorff dimension is bounded above by $n-2$.
Lin \cite{Lin1989CPAM} established the existence theorem for area-minimizing locally rectifiable $n$-currents and area-minimizing flat chains modulo $p$ with $p\geq 2$ in hyperbolic space. In the same work, Lin derived the ``boundary regularity at infinity'' result for area-minimizing flat chains modulo $2$.

The first aim of this work is to resolve the open problems raised by Lin in \cite{Lin1989CPAM} regarding whether $\operatorname{supp}(T)\cup\Gamma$ is smooth near $\Gamma$ in the Euclidean metric. In particular, Lin \cite{Lin1989CPAM} formulated the following assumption.

\begin{assumption}\label{assmp-normal}
Let $T$ be an area-minimizing locally rectifiable $n$-current in $\mathbb{H}^\fm$. Assume that $T$ is a normal current with respect to the standard Euclidean metric, and satisfies $\partial T = [\Gamma]$.
\end{assumption}

Under the foregoing assumption, we provide a positive solution to Problem 3 posed by Lin in \cite{Lin1989CPAM}.

\begin{theorem}[From locally rectifiable to \(C^{1,\alpha}\) regularity] \label{thm:rectifiable-C-1-alpha}
Let $T$ be an area-minimizing locally rectifiable $n$-current in $\mathbb{H}^\fm$, and let $\Gamma$ be a closed $C^{1,\alpha}$ submanifold of $\mathbb{R}^{\fm-1} \times \{0\}$ of dimension $n-1$, for some $0<\alpha\leq 1$. Assume that $T$ is a normal current with respect to the standard Euclidean metric and satisfies $\partial T = [\Gamma]$. Then there exists a constant $\rho_\Gamma>0$ such that, in the Euclidean metric, the restricted current
\[
T \mathbin{\llcorner} \bigl\{ (x', x^\fm )\in \mathbb{R}^\fm : x^\fm <\rho_\Gamma \bigr\}
\]
admits a representation as an $n$-dimensional $C^{1,\alpha}$ submanifold of $\mathbb{R}^\fm$ up to the boundary $\Gamma$.

Moreover, at every point $P\in \Gamma$, the Euclidean tangent plane of $T$ at $P$ is vertical, meaning that it is orthogonal to the hyperplane $\{x^\fm=0\}$.
\end{theorem}

We also establish a local version of Theorem \ref{thm:rectifiable-C-1-alpha}. For any point $Q \in \Gamma$, let $T_Q\Gamma$ denote the $(n-1)$-dimensional tangent plane of $\Gamma$ at $Q$, where $Q\in \mathbb{R}^{\fm-1} \times \{0\}$; we naturally identify $T_Q\Gamma$ as a subset of $\mathbb{R}^{\fm-1} \times \{0\}$. We then introduce the $n$-dimensional vertical half-plane given by
\begin{align}\label{eq-H+-0}
    H^+ = \bigl\{(x', x^\fm)\in \mathbb{R}^\fm: (x',0) \in T_Q\Gamma,\ x^\fm \in \mathbb{R}^+\bigr\}.
\end{align}
In Assumption \ref{assmp-local} below, we normalize coordinates by setting
\begin{align}\label{eq-Q-coor}
    Q= (\mathbf{0},0),
\end{align}
and fix the associated half-plane $H^+$. To be precise, we impose the coordinate condition on $H^+$ that
\begin{align}\label{eq-H-coor}
    x^n = x^{n+1} = \cdots = x^{\fm-1}=0.
\end{align}
For any $R>0$, we define the open set
\begin{align}\label{eq-GR+-0}
    G_R &= \bigl\{(x', x^\fm)\in \mathbb{R}^\fm: |x'|<R,\ 0<x^\fm <R\bigr\},
\end{align}
and the corresponding restricted domain in $H^+$ by
\begin{align}\label{eq-HR+-0}
    B_R^+ &= G_R\cap H^+.
\end{align}

\begin{assumption}\label{assmp-local}
Let $T$ be an area-minimizing locally rectifiable $n$-current in $\mathbb{H}^\fm$ that is asymptotic to $\Gamma$, which is a closed $(n-1)$-dimensional $C^{1}$ submanifold of $\mathbb{R}^{\fm-1} \times \{0\}$ and
\begin{align}\label{local-C1alpha}
    \Gamma \cap \overline{G_R}\in C^{1,\alpha},
\end{align}
for some $0<\alpha\leq 1$.
Let $H^+, G_R, B_R^+$ be defined as in \eqref{eq-H+-0}, \eqref{eq-GR+-0}, and \eqref{eq-HR+-0}, respectively, with the coordinate condition \eqref{eq-H-coor} satisfied on $H^+$. There exists a constant $R>0$ such that the following hold:

(1) For a single $r\in (0,R)$,
\begin{align}\label{eq-Proj-0}
     \operatorname{Proj}\bigl(T \mathbin{\llcorner} G_r\bigr)  = [B^+_r],
 \end{align}
 where $\operatorname{Proj}\bigl(T \mathbin{\llcorner} G_r\bigr)$ denotes the Euclidean orthogonal projection of $T \mathbin{\llcorner} G_r$ onto $H^+$;

(2) For a fixed small constant $c_0 = c_0(\fm, \Gamma) > 0$, and for all $P \in G_R$, the hyperbolic mass satisfies
\begin{align}\label{eq-mass-assump}
     M_{\mathbb{H}^\fm}\bigl(T \mathbin{\llcorner} B_{\mathbb{H}^\fm}(P, 2)\bigr) < e^{c_0 \bigl(x^\fm(P)\bigr)^{-\alpha}}.
\end{align}
\end{assumption}

We note that Assumption \ref{assmp-local} is weaker than Assumption \ref{assmp-normal}. By the Constancy Theorem (4.1.7, \cite{Federer}), there exists a small constant $\rho_\Gamma>0$ depending on $\Gamma$ such that the validity of \eqref{eq-Proj-0} for a single $r\in(0,\rho_\Gamma)$ is equivalent to its validity for all $r\in(0,\rho_\Gamma)$. Equation \eqref{eq-Proj-0} enforces the multiplicity one condition; this serves to exclude singularities, such as branch points, which occur in higher multiplicity scenarios.

As $x^\fm$ decays exponentially with respect to the geodesic distance function on $\mathbb{H}^\fm$ as the latter tends to infinity, the term
\begin{align}
    e^{c_0 \bigl(x^\fm(\cdot)\bigr)^{-\alpha}}
\end{align}
in \eqref{eq-mass-assump} grows \textbf{doubly exponentially} with respect to the geodesic distance function on $\mathbb{H}^\fm$. A key insight in our work is that the doubly exponential mass growth condition implies a uniform upper bound for the local mass near the asymptotic boundary.

\begin{theorem}\label{thm-local-C1a}
There exists a constant $\rho_\Gamma > 0$ such that if Assumption \ref{assmp-local} holds for some $R \in (0, \rho_\Gamma]$, then  for any $r \in (0, R)$, $\operatorname{supp} (T) \cap G_r$ is an $n$-dimensional analytic submanifold of $G_r$ in the Euclidean metric. This submanifold extends continuously to $\Gamma \cap \overline{G_r}$ and is of class $C^{1,\alpha}$ up to this boundary.

Moreover, at every point $P \in \Gamma \cap \overline{G_r}$, the Euclidean tangent plane of $T$ at $P$ is vertical, meaning it is orthogonal to the hyperplane $\{x^\fm = 0\}$.
\end{theorem}

We next introduce the corresponding system of equations and proceed to investigate its regularity theory; in general, however, the analysis of such regularity theory poses substantial challenges owing to the absence of a maximum principle.
Examples demonstrate that solutions of the system of minimal surface equations may fail to exist, be non-unique, or lack stability; for instance, even when the domain is a 4-dimensional ball and the boundary values are analytic, Lipschitz solutions to the minimal surface system generally do not exist. While Lawson and Osserman \cite{LO} established the existence of solutions to the minimal surface system when the domain is 2-dimensional, they also constructed examples showing that such solutions are generally neither unique nor stable.
Within the hyperbolic setting, we have rigorously derived an elegant system of equations governing minimal surfaces near $\Gamma$, which we present as follows.

Locally in a neighborhood of $Q\in \Gamma$, we introduce the Euclidean coordinate chart
\begin{align}\label{eq-y}
y=(y', y^n)=(y^1, \cdots, y^{n-1}, y^n)= (x^1,\cdots, x^{n-1}, x^\fm)
\end{align}
on $H^+$. By virtue of \eqref{eq-Q-coor}, $Q$ corresponds to the origin in the $y$-coordinate system.

\begin{lemma}\label{lem-Main}
Let $T$ be an area-minimizing locally rectifiable current in $\mathbb{H}^\fm$. Suppose that near $Q\in \Gamma$, $T$ can be represented as the graph of a $C^1$ mapping
\begin{align}
\mathbf{u}(y) = \bigl(u_1(y), \cdots, u_{\fm-n}(y)\bigr)
\end{align}
on some domain $\Omega \subseteq H^+$. Then $\mathbf{u}$ is real-analytic and satisfies the system of equations
\begin{align}\label{eq-main-0}
g^{ij}\frac{\partial^2 u_s}{\partial y^i \partial y^j} - \frac{n}{y^{n}} \frac{\partial u_s}{\partial y^{n}} =0 \quad \text{in } \Omega,
\end{align}
for $s=1, \cdots, \fm-n$. Here, the coefficients $g_{ij}$ are defined by
\begin{align}\label{eq-gij}
g_{ij}:=\delta_{ij} +\sum_{l=1}^{\fm -n} \frac{\partial u_l}{\partial y^i}\frac{\partial u_l}{\partial y^j},
\end{align}
and $(g^{ij})$ denotes the inverse matrix of $(g_{ij})$.
\end{lemma}

The second aim of this work is to settle the PDE aspects of the asymptotics for area-minimizing locally rectifiable currents in hyperbolic space. We seek to relate the geometric properties of the asymptotic boundary $\Gamma$ to the analytic properties of the asymptotic behavior of $T$.

\begin{theorem}[From $C^{1,\alpha}$ to $C^{n,\alpha}$]\label{thm-Cna}
There exists a constant $\rho_\Gamma>0$ such that if Assumption \ref{assmp-local} holds and
\begin{align}
\Gamma \cap \overline{G_R} \in C^{n,\alpha},
\end{align}
for some $R\in (0, \rho_\Gamma]$ and $\alpha\in (0,1)$, then for any $r\in (0, R)$, $\operatorname{supp}(T) \cap G_r$ is an $n$-dimensional analytic submanifold of $G_r$ in the Euclidean metric. This submanifold extends continuously to $\Gamma \cap \overline{G_r}$ and is of class $C^{n,\alpha}$ up to this boundary.
\end{theorem}

\begin{theorem}[Boundary Regularity Theorem I]\label{thm-BRT1}
There exists a constant $\rho_\Gamma>0$ such that if Assumption \ref{assmp-local} holds and
\begin{align}
\Gamma \cap \overline{G_R} \in C^{\infty},
\end{align}
for some $R\in (0, \rho_\Gamma]$, then for any $r\in (0, R)$, $\operatorname{supp}(T) \cap G_r$ is the graph of an analytic $(\fm-n)$-valued function $\mathbf{u}$ defined on $B^+_r$ in the Euclidean metric. Moreover, $\mathbf{u}$ can be regarded as a smooth function of $y'$, $y^n$, and $y^n\log(y^n)$ on the closed domain
\begin{align}
\bigl\{(y', y^n, y^n\log(y^n)): |y'|\leq r,\ 0\leq y^n \leq r,\ 0\leq  y^n|\log(y^n)|\leq  r\bigr\}.
\end{align}
If in addition $n$ is even, then $\mathbf{u}$ is of class $C^\infty$ in $y$ on $\overline{B^+_r}$.
\end{theorem}
By Theorem \ref{thm-BRT1}, we derive the Taylor expansion of $\mathbf{u}$ with respect to $y'$, $y^n$, and $y^n\log(y^n)$, given by
\begin{align}\label{eq-u-exp-0}
    \mathbf{u} = \varphi(y') +\sum_{i=2}^{n}\mathbf{c}_i(y')(y^n)^i +\sum_{i=n+1}^k\sum_{j=0}^{\lfloor \frac{i-1}{n}\rfloor} \mathbf{c}_{i,j}(y')(y^n)^i(\log(y^n))^j +\mathbf{R}_k,
\end{align}
for any integer $k\geq n+1$, in the sense of Definition \ref{def-exp} below. If $n$ is even, all logarithmic terms in \eqref{eq-u-exp-0} vanish, reducing \eqref{eq-u-exp-0} to a standard Taylor expansion. The coefficients in \eqref{eq-u-exp-0} are determined via formal computations as detailed in the proof of Lemma \ref{lem-Form-Comp}.

In fact, for $s \in \{1,\cdots, \fm-n\}$, the $s$ component of $\mathbf{c}_2$ is
\begin{align}
    c_{s,2}(y') = \frac{1}{2}\sqrt{1+|D\varphi_s(y')|^2} \, \big(H(y'), \nu_s(y')\big),
\end{align}
where $H(y')$ is the mean curvature vector of $\Gamma$ at $(y',\varphi(y'))$, and $\nu_s$ is a unit normal vector in the form of
\begin{align}
    \nu_s = \frac{(-D\varphi_s, \mathbf{e}_s)}{\sqrt{1+|D \varphi_s|^2}}.
\end{align}

If $n=3$, then the $s$ component of $\mathbf{c}_{4,1}$ is
\begin{align}\label{eq-cs41}
    c_{s, 4,1} = -\frac{1}{8}\sqrt{1+|D\varphi_s(y')|^2} \left( \Delta^\perp H-2|H|^2 H + Q(\operatorname{II})(H), \nu_s \right),
\end{align}
where $\operatorname{II}$ is the second fundamental form of $\Gamma$ and, with $g$ given by \eqref{eq-gij},
\begin{align}
    Q(\operatorname{II})(H) = g^{ik}g^{jl} ( \operatorname{II}_{ij}, H )_g\operatorname{II}_{kl}.
\end{align}
  Since the equation
\begin{align}
    \Delta^\perp H - 2|H|^2 H + Q(\operatorname{II})(H) = 0
\end{align}
characterizes Willmore surfaces, we  obtain the following corollary.
\begin{corollary}
Let \(n = 3\). Suppose $\Gamma \in C^{4,\alpha}$ for some $0<\alpha<1.$ There exists a constant \(\rho_\Gamma > 0\) such that whenever Assumption~\ref{assmp-local} is satisfied on an arbitrary local region $G_R$, with $R\in (0, \rho_\Gamma)$,  the following statements are equivalent:
\begin{itemize}
    \item \(\Gamma\) is a Willmore surface, i.e., a critical point of the Willmore functional
    \[
    \int_\Gamma |H|^2 \, d\sigma_g;
    \]

    \item The submanifold \(\operatorname{supp}(T) \cap \bigl\{x\in \mathbb{R}^\fm_+ : x^\fm < \rho_\Gamma\bigr\}\) is analytic up to \(\Gamma\) with respect to the Euclidean metric;

    \item The submanifold \(\operatorname{supp}(T) \cap \bigl\{x\in \mathbb{R}^\fm_+ : x^\fm < \rho_\Gamma\bigr\}\) is \(C^{4}\) up to \(\Gamma\) with respect to the Euclidean metric.
\end{itemize}
\end{corollary}

In the case where $\Gamma$ has finite regularity, we establish the following boundary regularity theorem. This result implies the validity of the expansion \eqref{eq-u-exp-0} for $n+1 \leq k \leq l$ whenever $\Gamma \cap \overline{G_R} \in C^{l,\alpha}$, even though \eqref{eq-Rk-Tan} naturally fails to hold.

\begin{theorem}[Boundary Regularity Theorem II]\label{thm-BRT2}
There exists a constant $\rho_\Gamma>0$ such that if Assumption \ref{assmp-local} holds and
\begin{align}
\Gamma \cap \overline{G_R} \in C^{l,\alpha},
\end{align}
for some $R\in (0, \rho_\Gamma], 0<\alpha<1$, and some integer $l\geq 1$, then for any $r\in (0, R)$, $\operatorname{supp}(T) \cap G_r$ is the graph of an analytic $(\fm-n)$-valued function $\mathbf{u}$ defined on $B^+_r$ in the Euclidean metric. Moreover, for $m<\frac{l}{n}\leq m+1$, there exist $(\fm-n)$-valued functions
\begin{align}
    \mathbf{w}_0, \mathbf{w}_1, \cdots, \mathbf{w}_m \in C^{l,\epsilon}(\overline{B^+_r}) \quad \text{for all } \epsilon\in (0, \alpha),
\end{align}
such that
\begin{align}
    \mathbf{u} =  \mathbf{w}_0 + \mathbf{w}_1 \log (y^n) + \cdots + \mathbf{w}_m (\log (y^n))^m \quad \text{in } \overline{B^+_r},
\end{align}
and for each $j=1,\cdots, m$,
\begin{align}
    \partial_n^i \mathbf{w}_j(y', 0)=0 \quad \text{for } (y', 0) \in \overline{B^+_r} \text{ and all } 0\leq i\leq jn.
\end{align}
If in addition $n$ is even, or if
\begin{align}
\partial_n^{n+1} \mathbf{w}_1(y', 0)=0 \quad \text{for } (y', 0) \in \overline{B^+_r},
\end{align}
then $\mathbf{u} \in C^{l, \epsilon}(\overline{B^+_r})$.
\end{theorem}
A natural question arises as to whether the series
\begin{align}\label{eq-u-series}
\varphi +\sum_{i=2}^{n}\mathbf{c}_i(y')(y^n)^i +\sum_{i=n+1}^\infty\sum_{j=0}^{\lfloor \frac{i-1}{n}\rfloor} \mathbf{c}_{i,j}(y')(y^n)^i(\log(y^n))^j,
\end{align}
which corresponds to the expansion of a real solution $\mathbf u$, converges uniformly in $\overline{B^+_r}$. By the work of Kichenassamy \cite{K2} and Kichenassamy and Littman \cite{KL1, KL2}, the answer is affirmative provided that $\varphi$ and $\mathbf{c}_{n+1,0}$ in \eqref{eq-u-exp-0} are real-analytic (see also \cite{K1}). However, given an arbitrary real solution $\mathbf{u}$, it remains unknown whether the corresponding coefficient $\mathbf{c}_{n+1,0}$ is analytic. Han and Jiang \cite{HanJiang1} investigated this problem in the hypersurface setting. For the high-codimension case, we establish an analogous convergence theorem.

\begin{theorem}[Convergence Theorem]\label{thm-Conv-1}
There exists a constant $\rho_\Gamma>0$ such that if Assumption \ref{assmp-local} holds and
\begin{align}
\Gamma \cap \overline{G_R} \in C^\omega,
\end{align}
for some $R\in (0, \rho_\Gamma]$, then for any $r\in (0, R)$, $\operatorname{supp}(T) \cap G_r$ is the graph of an analytic $(\fm-n)$-valued function $\mathbf{u}$ defined on $B^+_r$ in the Euclidean metric. Moreover, $\mathbf{u}$ admits an analytic representation in terms of $y$ and $y^n\log(y^n)$ on the set
\begin{align}
    \bigl\{(y', y^n, y^n\log(y^n)): |y'|\leq r,\ 0\leq y^n \leq r,\ 0\leq y^n|\log(y^n)|\leq r\bigr\}.
\end{align}
In particular, if $n$ is even, then $\mathbf{u}$ is analytic in $y$ on $\overline{B^+_r}$.
\end{theorem}

If $\Gamma$ is only locally smooth (and non-analytic), we state a convergence theorem for $\mathbf{u}$ being analytic in $(y^n)^n\log(y^n)$; see Theorem \ref{thm-conv-2} below.

We conclude the introduction with a brief outline of the paper. In Section \ref{sec-pre}, we clarify certain concepts referenced throughout this work. In Section \ref{sec:rectifiable-C-1-alpha}, we establish a local mass bound estimate for $T$ and prove Theorems \ref{thm:rectifiable-C-1-alpha} and \ref{thm-local-C1a}. In Section \ref{sec-PDE}, we derive the system of minimal surface equations \eqref{eq-main-0} using two distinct methods. In Section \ref{sec-HR}, we prove Theorems \ref{thm-Cna}, \ref{thm-BRT1}, and \ref{thm-BRT2}. In Section \ref{sec-coeff}, we compute the expansion coefficients $\mathbf c_2$ and $\mathbf c_{4,1}$. In Section \ref{sec-Conv}, we prove Theorem \ref{thm-Conv-1} and present a convergence theorem in the smooth, non-analytic setting.

We thank Fanghua Lin and Hans-Joachim Hein for helpful discussions.

\section{Preliminary}\label{sec-pre}
In this paper, we say that $T$ is an area-minimizing locally rectifiable $n$-current in $\mathbb{H}^\fm$ if, for every bounded domain $U \subseteq \mathbb{H}^\fm$, the restricted current $T\, \mathbin{\llcorner}\, U$ is absolutely area-minimizing. Precisely, for any rectifiable current $S$ in $\mathbb{H}^\fm$ satisfying $\partial S = \partial (T\, \mathbin{\llcorner}\, U)$, the mass inequality
\begin{align}
    M(T \mathbin{\llcorner} U) \leq M(S)
\end{align}
holds.

Throughout this paper, we make the standing assumption that $\Gamma$ is an oriented closed $C^{1,\alpha}$ submanifold of the hyperplane $\mathbb{R}^{\fm-1} \times \{0\}$ at infinity, for some $0 < \alpha \leq 1$. We say that $T$ is asymptotic to $\Gamma$ if $\partial T = 0$ in the hyperbolic space $\mathbb{H}^\fm$, and the boundary of $\operatorname{supp}(T)$ coincides with $\Gamma$ in the Euclidean metric.

Let $\nu_\Gamma$ denote a unit normal vector field on $\Gamma$ in $\mathbb{R}^{\fm-1}$. For any $Q\in \Gamma$ and $r>0$, we define
\begin{align}
    \delta(Q, r) = \min \bigl\{ \operatorname{dist}\bigl(Q+r\nu_\Gamma(Q), \Gamma\bigr), \operatorname{dist}\bigl(Q-r\nu_\Gamma(Q), \Gamma\bigr) \bigr\}.
\end{align}

For $x\in \mathbb{R}^{\fm-1}\times \{0\}$, we set
\begin{align}
    d(x) = \operatorname{dist}(x, \Gamma).
\end{align}
The minimal surface $\operatorname{supp}(T)$ is contained in the set
\begin{align}\label{eq-W}
    W := \mathbb{R}^{\fm-1}\times (0, \infty)  \setminus \bigcup_{\substack{d(x)>0, \\ x^\fm=0}} B_{\mathbb{R}^\fm}(x,d(x)).
\end{align}

For a multi-valued function $\mathbf{w}(x) = \bigl(\dots, w_s(x), \dots\bigr)$, we adopt the notation
\begin{align}\label{eq-derivative}
    w_{s, i} = \frac{\partial w_s}{\partial x^i}, \quad w_{s, ij} = \frac{\partial^2 w_s}{\partial x^i \partial x^j}.
\end{align}

\begin{lemma}\label{lem-rho-Gamma}
Let $\Gamma$ be of class $C^{1,\alpha}$ for some $0<\alpha\leq 1$ and let $W$ be defined as in \eqref{eq-W}. Then there exists a small constant $\rho_\Gamma>0$ such that for any fixed $Q\in \Gamma$, if
\begin{align}
x=(\tilde x, x^\fm) \in W \cap \{x^\fm <\rho_\Gamma \}    
\end{align}
 and
\begin{align}\label{eq-a-r}
(\tilde x,0) = Q+r\nu_\Gamma(Q)
\end{align}
for some $r>0$ and some unit normal vector $\nu_\Gamma(Q)$ to $\Gamma$ at $Q$, then
\begin{align}\label{eq-delta-1}
r < C (x^\fm)^{1+\alpha},
\end{align}
where $C = C(\fm,n,\Gamma)$ is a constant independent of $Q$ and $x$.
\end{lemma}

\begin{proof}
Translate $Q$ to the origin and adopt the coordinate conventions from \eqref{eq-H+-0} to \eqref{eq-H-coor}. By a suitable rotation of coordinates, we may set
\begin{align}
\nu_\Gamma(Q) = (x^1, \dots, x^{n-1}, x^n, x^{n+1}, \dots, x^\fm) = (0, \dots, 0, 1, 0, \dots, 0).
\end{align}
For any point $P\in \Gamma$ of the form
\begin{align}\label{eq-P}
(x', \varphi(x'), 0) := (x^1, \dots, x^{n-1}, \varphi (x^1, \dots, x^{n-1}), 0),
\end{align}
the Euclidean distance between $s\nu_\Gamma(Q)$ and $P$ is given by
\begin{align}
\operatorname{dist}(s\nu_\Gamma(Q), P) = \sqrt{|x'|^2 + (\varphi_1(x')-s)^2 + \sum_{\beta=2}^{\fm-n}\varphi_\beta^2(x')},
\end{align}
where $\varphi_i$ denotes the $i$-th component of $\varphi$. Varying $P$,  at a minimizer of this distance we have, for all $i=1,\dots, n-1$ that
\begin{align}\label{eq-crit}
x_i + (\varphi_1 - s)\varphi_{1, i} + \sum_{\beta=2}^{\fm-n}\varphi_\beta \varphi_{\beta,i} = 0.
\end{align}
Trivially,
\begin{align}
\min_{P\in \Gamma}\operatorname{dist}(s\nu_\Gamma(Q), P) \leq s.
\end{align}
There exists a small constant $\delta = \delta(\Gamma) > 0$ such that for all $s\in (0, \delta]$, the distance between $s\nu_\Gamma(Q)$ and $\Gamma$ is attained by $\operatorname{dist}(s\nu_\Gamma(Q), P)$ for some $P\in \Gamma$ of the form \eqref{eq-P} satisfying the critical condition \eqref{eq-crit}. By the $C^{1,\alpha}$ regularity of $\Gamma$, we have the estimates
\begin{align}\label{eq-s-1}
|\varphi(x')| \leq C|x'|^{1+\alpha}, \quad |D\varphi(x')| \leq C|x'|^{\alpha}.
\end{align}
Condition \eqref{eq-crit} then implies
\begin{align}
|x' + (\varphi_1(x') - s)D\varphi_{1}(x')| \leq C|x'|^{1+2\alpha}.
\end{align}
Combining this with the bound
\begin{align}
|\varphi_1(x') D\varphi_{1}(x')| \leq C|x'|^{1+2\alpha},
\end{align}
we deduce
\begin{align}\label{eq-s-2}
|x'| \leq C s |D\varphi_{1}(x')|.
\end{align}
From \eqref{eq-s-1} and \eqref{eq-s-2}, it follows that
\begin{align}\label{eq-s-3}
s \geq C^{-1}|x'|^{1-\alpha}.
\end{align}
In the special case $\alpha = 1$, by shrinking $\delta$ if necessary, \eqref{eq-s-1} and \eqref{eq-s-2} force $x' = 0$ and hence $\operatorname{dist}(s\nu_\Gamma(Q), P) = s$.

By the definition of $W$, for each $s>0$ we subtract a ball of radius $\operatorname{dist}(s\nu_\Gamma(Q), \Gamma)$ from the upper half-space. If
\begin{align}
\operatorname{dist}(s\nu_\Gamma(Q), \Gamma) = \operatorname{dist}(s\nu_\Gamma(Q), P) \geq x^\fm,
\end{align}
then algebraic manipulation yields
\begin{align}
r &< s - \sqrt{\operatorname{dist}(s\nu_\Gamma(Q), P)^2 - (x^\fm)^2} \notag \\
&\leq s - \sqrt{|x'|^2 + (\varphi_1(x')-s)^2 + \sum_{\beta=2}^{\fm-n}\varphi_\beta^2(x') - (x^\fm)^2} \notag \\
&\leq \frac{2\varphi_1(x') s + (x^\fm)^2}{s + \sqrt{|x'|^2 + (\varphi_1(x')-s)^2 + \sum_{\beta=2}^{\fm-n}\varphi_\beta^2(x') - (x^\fm)^2}} \notag \\
&\leq \frac{2\varphi_1(x') s + (x^\fm)^2}{s}.
\label{eq-r-s}
\end{align}
For $\alpha = 1$, we have $x' = 0$. Setting $s = \delta$ gives
\begin{align}
r < \delta^{-1}(x^\fm)^2.
\end{align}
For $\alpha \in (0,1)$, we set
\begin{align}\label{eq-s-xm}
    s = (x^\fm)^{1-\alpha}.
\end{align}
 There exists a small constant $\rho_\Gamma>0$ such that $x^\fm \in (0, \rho_\Gamma]$ implies $s \in (0, \delta]$. By \eqref{eq-s-3} and \eqref{eq-s-xm}, we obtain the key bound
\begin{align}\label{eq-s-31}
|x'| \leq C x^\fm.
\end{align}
Combining \eqref{eq-s-1}, \eqref{eq-s-31}, and \eqref{eq-r-s}, we conclude
\begin{align}
r < C (x^\fm)^{1+\alpha}.
\end{align}
Since $\Gamma$ is closed, we may enlarge $C$ if necessary to ensure it is independent of the choice of $Q\in \Gamma$.
\end{proof}

\begin{remark}
    From the proof of Lemma \ref{lem-rho-Gamma}, we deduce that for $\alpha\in (0,1)$,
    \begin{align}\label{eq-r-delta-1}
    \sup_{Q\in \Gamma}\left(1 - r^{-1}\delta(Q, r)\right) \leq C_\Gamma r^{\frac{2\alpha}{1-\alpha}},
    \end{align}
    or equivalently,
\begin{align}\label{eq-r-delta-2}
    r - C_\Gamma r^{\frac{1+\alpha}{1-\alpha}} \leq \delta(Q, r) \leq r.
    \end{align}
    If $\alpha=1$, then $\delta(Q,r) = r$. See also Section 1 of \cite{Hardt&Lin1987}.
\end{remark}

\section{From Local Rectifiability to $C^{1,\alpha}$ Regularity}\label{sec:rectifiable-C-1-alpha}

For $r>0$, we define the domain
\begin{align}
    D_r = \bigl\{x\in \mathbb{R}^\fm: (x^1)^2+\cdots+(x^n)^2 < r^2,\ x^{n+1} = \cdots = x^\fm=0\bigr\}.
\end{align}

\begin{lemma}[Mass Bound I]\label{lem-massbound1}
Let $T$ be an absolutely area-minimizing rectifiable $n$-current in $\mathbb{R}^\fm$, and suppose that
\begin{align}
   \partial T \mathbin{\llcorner} B_{\mathbb{R}^\fm}(2) = 0,
\end{align}
where \( B_{\mathbb{R}^\fm}(2) \subseteq \mathbb{R}^\fm \) denotes the Euclidean ball of radius $2$ centered at the origin. If in addition the following two conditions hold:

(1) there exists a small constant $\delta>0$ such that the Gromov–Hausdorff distance between
\begin{align}
    \operatorname{supp}\bigl(T \mathbin{\llcorner} B_{\mathbb{R}^\fm}(2)\bigr)
\end{align}
and $D_2$ is bounded above by $\delta$;

(2) the projection identity
\begin{align}
     \operatorname{Proj}(T) \mathbin{\llcorner} D_{2-\delta} = [D_{2-\delta}]
\end{align}
holds, where $\operatorname{Proj}(T)$ denotes the orthogonal projection of $T \mathbin{\llcorner} B_{\mathbb{R}^\fm}(2)$ onto $D_3$;

then there exists a positive constant $c_0 = c_0(\fm, n)$ such that
\begin{align}
    M_{\mathbb{R}^\fm}\bigl(T \mathbin{\llcorner} B_{\mathbb{R}^\fm}(1)\bigr)
    &\leq e^{-\frac{1}{c_0\delta}} M_{\mathbb{R}^\fm}\bigl(T \mathbin{\llcorner} B_{\mathbb{R}^\fm}(2)\bigr) + C(n).
\end{align}
\end{lemma}

\begin{proof}
For $r\in [1, 2]$, we set
\begin{align}
    f(r) = M(T_r) :=  M\bigl(T \mathbin{\llcorner} B_{\mathbb{R}^\fm}(r)\bigr).
\end{align}

    For almost every $r\in [1,2]$, the boundary current $\partial\bigl(T \mathbin{\llcorner} B_{\mathbb{R}^\fm}(r)\bigr)$ is rectifiable. For each fixed such $r$, the support of $\partial\bigl(T \mathbin{\llcorner} B_{\mathbb{R}^\fm}(r)\bigr)$ is contained in the union of a countable collection of $(n-1)$-dimensional Lipschitz submanifolds $\{S_i\}_{i\geq 1}$ and a set $S_0$ of $\mathcal{H}^{n-1}$-measure zero. In addition, by partitioning the support into sufficiently fine pieces $\{S_i\}$ (modulo a set of $\mathcal{H}^{n-1}$-measure zero), we ensure the multiplicity is constant on each $S_i$, and also
\begin{align}
M\bigl(\partial\bigl(T \mathbin{\llcorner} B_{\mathbb{R}^\fm}(r)\bigr)\bigr) = \sum_i \theta_i \cdot \operatorname{Area}(S_i),
\end{align}
where $\theta_i \in \mathbb{N}$ is the constant multiplicity of the current restricted to $S_i$.
    For each $S_i$, we connect points on $S_i$ to the flat disk $D_{r-\delta}$ via shortest line segments. This construction defines a rectifiable $n$-current $R$ satisfying
    \begin{align}
        R \mathbin{\llcorner} D_{r-\delta} = [D_{r-\delta}], \,\,
        \partial R = \partial\bigl(T_r).
    \end{align}
    Moreover, there exists a positive constant $c_0>0$ such that
    \begin{align}
        M(T_r) \leq M(R) \leq c_0 \delta M(\partial T_r) + C(n),
    \end{align}
    where $C(n)$ denotes the mass of $D_2$. Following the argument in the proof of Lemma 1.3 in \cite{Lin1989CPAM}, slicing theory—with $u(r)$ denoting the distance function to the center of $D_2$—yields
\begin{align}\label{eq-f1}
    f'(r) \geq M\langle T_r, u, r^+\rangle = M(\partial T_r).
\end{align}
This implies the differential inequality
\begin{align}
    f(r) \leq c_0\delta f'(r) + C(n),
\end{align}
from which we deduce that for all $r\in [1,2]$,
\begin{align}
    f(r) \leq e^{\frac{r-2}{c_0\delta}}f(2) +C(n).
\end{align}
This completes the proof of the lemma.
\end{proof}

The next is a hyperbolic version of the above lemma.

\begin{lemma}[Mass bound II]\label{lem-massbound2}
Let $H^+$ be defined as in \eqref{eq-H+-0} and let $\{\tilde D_r\}$ be a family of $n$-dimensional hyperbolic disks of radius $r>0$ in $H^+$, sharing a common center $P$. Let $T$ be an area-minimizing locally rectifiable $n$-current in $\mathbb{H}^\fm$, and suppose that
\begin{align}
   x^\fm(P)=1, \quad \partial T \mathbin{\llcorner} B_{\mathbb{H}^\fm}(P, 2) =0. 
\end{align}
If in addition the following hold:

(1) there exists a small constant $\delta>0$ such that the Gromov–Hausdorff distance between
\begin{align}
    \operatorname{supp}\bigl(T \mathbin{\llcorner} B_{\mathbb{H}^\fm}(P, 2)\bigr)
\end{align}
and $\tilde D_2$ is bounded by $\delta$;

(2)
\begin{align}\label{eq-Proj1}
     \operatorname{Proj}(T) \mathbin{\llcorner} \tilde{D}_{2-\delta} = [\tilde{D}_{2-\delta}],
\end{align}
where $\operatorname{Proj}(T)$ denotes the Euclidean orthogonal projection of $T \mathbin{\llcorner} B_{\mathbb{H}^\fm}(P, 2)$ onto $H^+$;

then there exists a positive constant $c_0 = c_0(\fm, n)$ such that
\begin{align}
    M_{\mathbb{H}^\fm}\bigl(T \mathbin{\llcorner} B_{\mathbb{H}^\fm}(P, 1)\bigr)
    &\leq e^{-\frac{1}{c_0\delta}} M_{\mathbb{H}^\fm}\bigl(T \mathbin{\llcorner} B_{\mathbb{H}^\fm}(P, 2)\bigr) + C(n).
\end{align}
\end{lemma}
\begin{proof}
According to the proof of Lemma 1.3 in Lin \cite{Lin1989CPAM}, a hyperbolic area may be described by the parametric integrand
    \begin{align}
        \Phi(x,\xi) = |x^\fm|^{-n}|\xi|.
    \end{align}
    In $B_{\mathbb{H}^\fm}(P, 2)$, this integrand is (with respect to the Euclidean metric) elliptic with an ellipticity bound $e^{-2n}$.

Similarly as in the proof of Lemma \ref{lem-massbound1}, we construct a rectifiable $n$-current $R$, whose Euclidean and hyperbolic masses are bounded by $c_0\delta M(\partial (T \,\Large\llcorner\normalsize \, B_{\mathbb{R}^\fm}(P, r))) + C(n)$.
The rest of the proof is the same as that of Lemma \ref{lem-massbound1}.
\end{proof}

\begin{theorem}\label{thm-mass}
Let $\Gamma$ be a $C^{1,\alpha}$ submanifold of the hyperplane at infinity, let $\rho_\Gamma$ be defined as in Section \ref{sec-pre} and let $P$ be a point in $G_r$ with
\begin{align}\label{eq-xm-P}
    r < \tfrac{1}{2}\rho_\Gamma.
\end{align}Assume that $T$ is an area-minimizing locally rectifiable $n$-current in $\mathbb{H}^\fm$, and $T$ is a normal current in the standard Euclidean metric with $\partial T = [\Gamma]$. Then 
\begin{align}
    M_{\mathbb{H}^\fm}\bigl(T \mathbin{\llcorner} B_{\mathbb{H}^\fm}(P, 1)\bigr) < C\bigl(M_{\mathbb{R}^\fm}(T) +1\bigr),
\end{align}
where the constant $C = C(\fm,\Gamma)$ is independent of the choice of  $P$.
\end{theorem}

\begin{proof}
    Let $\rho_\Gamma$ be defined as in Section \ref{sec-pre}. We prove the theorem without applying Anderson's monotonicity theorem in \cite{Anderson1982}.

First, we establish \eqref{eq-Proj1}. Consider the current $S$ represented by the set
\begin{align}
    \bigl\{(x', x^\fm) \in \mathbb{R}^\fm: (x',0)\in \Gamma, \, x^\fm\leq 0\bigr\},
\end{align}
satisfying $\partial S = - [\Gamma]$. By our assumptions, we have
\begin{align}
    \partial (T+S) =0
\end{align}
in the standard Euclidean metric, and for every compact domain $K\subseteq \mathbb{R}^\fm$, the restricted current
\begin{align}
    (T+ S) \mathbin{\llcorner} K
\end{align}
is a normal current in $\mathbb{R}^\fm$. We fix a constant
\begin{align}
    \epsilon= \tfrac{1}{2}\rho_\Gamma.
\end{align}
For a fixed point $Q \in \Gamma$ and the associated half-space $H^+$, we define the restricted current
\begin{align}
    T_\epsilon:=(T+S) \mathbin{\llcorner} B_{\mathbb{R}^\fm}(Q, \epsilon).
\end{align}
By the Constancy Theorem (4.1.7 in \cite{Federer}) and the structural property of the set $W$ in \eqref{eq-W}, the Euclidean orthogonal projection of $T_\epsilon$ onto $H^+$ satisfies
\begin{align}
  \operatorname{Proj}(T_\epsilon)\mathbin{\llcorner} \bigl(B_{\mathbb{R}^\fm}(Q, \epsilon/2) \cap H^+\bigr)=
    k\bigl[B_{\mathbb{R}^\fm}(Q, \epsilon/2) \cap H^+\bigr],
\end{align}
for some integer $k\in \mathbb{N}$. It is immediate that $k=1$ by the definition of $S$, which yields \eqref{eq-Proj1}.
    
    Secondly, we derive a mass bound for $T \mathbin{\llcorner} B_{\mathbb{H}^\fm}(P, 2)$. It is immediate that
    \begin{align}
        M_{\mathbb{R}^\fm}\bigl(T \mathbin{\llcorner} B_{\mathbb{H}^\fm}(P, 2)\bigr) \leq M_{\mathbb{R}^\fm}(T).
    \end{align}
    Within $B_{\mathbb{H}^\fm}(P, 2)$, we have $x^\fm \geq e^{-2}\cdot x^\fm(P)$, which implies the estimate
    \begin{align}\label{eq-mass1}
        M_{\mathbb{H}^\fm}\bigl(T \mathbin{\llcorner} B_{\mathbb{H}^\fm}(P, 2)\bigr) \leq e^{2n}\cdot \bigl(x^\fm(P)\bigr)^{-n} M_{\mathbb{R}^\fm}(T).
    \end{align}

    Lastly, we set
    \begin{align}
        \tilde \delta = \bigl(e^2 x^\fm(P)\bigr)^{1+\alpha}.
    \end{align}
    Then by \eqref{eq-delta-1}, the Gromov–Hausdorff distance between the support of $T \mathbin{\llcorner} B_{\mathbb{H}^\fm}(P, 2)$ and $\tilde D_2$ is bounded by $\tilde \delta$. We define an isometric map $\Phi$ by
    \begin{align}\label{eq-phi}
        \Phi (x', x^\fm) = \bigl(x^\fm(P)\bigr)^{-1}\bigl(x' - x'(P), x^\fm\bigr),
    \end{align}
    which maps $B_{\mathbb{H}^\fm}(P,2)$ bijectively onto $B_{\mathbb{H}^\fm}(\Phi(P),2)$ with $x^\fm(\Phi(P))=1$. We deduce that the Gromov–Hausdorff distance between the support of $(\Phi_\# T) \mathbin{\llcorner} B_{\mathbb{H}^\fm}(\Phi(P), 2)$ and $\Phi_\# \tilde D_2$ is bounded by
    \begin{align}\label{eq-delta-2}
        \delta = e^{2+2\alpha}\bigl(x^\fm(P)\bigr)^{\alpha}.
    \end{align}
    
    Therefore, by Lemma \ref{lem-massbound2}, together with \eqref{eq-mass1} and \eqref{eq-delta-2}, we obtain the chain of inequalities
    \begin{align}
        M_{\mathbb{H}^\fm}\bigl(T \mathbin{\llcorner} B_{\mathbb{H}^\fm}(P, 1)\bigr) & = M_{\mathbb{H}^\fm}\bigl((\Phi_\# T) \mathbin{\llcorner} B_{\mathbb{H}^\fm}(\Phi(P), 1)\bigr) \notag \\
        &\leq e^{-\frac{1}{c_0\delta}} M_{\mathbb{H}^\fm}\bigl((\Phi_\# T) \mathbin{\llcorner} B_{\mathbb{H}^\fm}(\Phi(P), 2)\bigr) + C(\fm) \notag \\
        &= e^{-\frac{1}{c_0\delta}} M_{\mathbb{H}^\fm}\bigl(T \mathbin{\llcorner} B_{\mathbb{H}^\fm}(P, 2)\bigr) + C(\fm) \notag \\
        &\leq e^{-\frac{1}{c_0\delta}} e^{2n}\cdot \bigl(x^\fm(P)\bigr)^{-n} M_{\mathbb{R}^\fm}(T) + C(\fm) \notag \\
        &\leq C\bigl(M_{\mathbb{R}^\fm}(T) +1\bigr),\label{eq-C-mass}
    \end{align}
    where $C$ in \eqref{eq-C-mass} depends on $\fm$ and $\Gamma$.
\end{proof}

As the proof of Theorem \ref{thm-mass} is essentially local, we in fact establish the following lemma.

\begin{lemma}
 Let $\Gamma$ be a $C^{1}$ submanifold of the hyperplane at infinity with \eqref{local-C1alpha} holds, let $\rho_\Gamma$ be defined as in Section \ref{sec-pre} and let $P\in G_r$ with $r < \tfrac{1}{2}\rho_\Gamma$. Assume that $T$ is an area-minimizing locally rectifiable $n$-current in $\mathbb{H}^\fm$ that is asymptotic to $\Gamma$, satisfying 
\begin{align}
    \partial T \mathbin{\llcorner} B_{\mathbb{H}^\fm}(P, 2) = 0,
\end{align}
and \eqref{eq-Proj1} with $\delta=\bigl(e^2 x^\fm(P)\bigr)^{1+\alpha}$. If in addition, for some small constant $c_0=c_0(\fm, \Gamma)>0$, the bound
\begin{align}
     M_{\mathbb{H}^\fm}\bigl(T \mathbin{\llcorner} B_{\mathbb{H}^\fm}(P, 2)\bigr) < e^{c_0 \bigl(x^\fm(P)\bigr)^{-\alpha}}
\end{align}
holds, then
\begin{align}
    M_{\mathbb{H}^\fm}\bigl(T \mathbin{\llcorner} B_{\mathbb{H}^\fm}(P, 1)\bigr) < C(\fm, n).
\end{align} 
\end{lemma}

Next, we proceed to prove Theorem \ref{thm:rectifiable-C-1-alpha}.

\begin{proof}[Proof of Theorem \ref{thm:rectifiable-C-1-alpha}]
    By Theorem \ref{thm-mass}, for every $P\in G_r$ with $r<\tfrac{1}{2}\rho_\Gamma$, we have
    \begin{align}
        M_{\mathbb{H}^\fm}\bigl(T \mathbin{\llcorner} B_{\mathbb{H}^\fm}(P, 1)\bigr) < C(\fm, \Gamma)\bigl(M_{\mathbb{R}^\fm}(T) +1\bigr).
    \end{align}
    Following the argument in the proof of Lemma 1.4 in \cite{Lin1989CPAM}, we employ the standard squashing deformation and Theorem 5.3.14 in \cite{Federer}. By applying the isometric map $\Phi$ defined in \eqref{eq-phi}, we obtain the interior regularity estimate that the restricted current
    \begin{align}
        \Phi_\# T \mathbin{\llcorner} B_{\mathbb{H}^\fm}(\Phi(P), 1)
    \end{align}
    is represented by the graph of a multi-valued function $\textbf{u}$ satisfying
    \begin{align}
        \|\textbf{u}\|_{C^{1,1}(\bar\Omega)} \leq \epsilon\bigl(x^\fm(P), \fm, \Gamma\bigr),
    \end{align}
    where $\Omega$ is the Euclidean projection of $B_{\mathbb{H}^\fm}(\Phi(P), 1)$ onto $H^+$, and
    \begin{align}
        \epsilon\to 0 \quad \text{as} \quad x^\fm(P) \to 0.
    \end{align}
    The remainder of the proof follows the same lines as the proofs of Theorem 2.2 and Theorem 3.1 in \cite{Hardt&Lin1987}.
\end{proof}

Theorem \ref{thm-local-C1a} follows by the same argument, as it is essentially local.

\section{System of Minimal Surface Equations}\label{sec-PDE}

Let $Q\in \Gamma$ and let $H^+$ be defined as in \eqref{eq-H+-0}. Without loss of generality, we assume that for all $(x',x^\fm )\in H^+$,
\begin{align}
    x^i = 0 \quad \text{if} \quad n\leq i < \fm.
\end{align}
We define the domain $G_R$ as in
\eqref{eq-GR+-0}.

By Theorem \ref{thm-local-C1a}, the area-minimizing rectifiable current $T$ admits a graphical representation near $Q$ by the vector-valued function $\mathbf{u} = (u_1,\cdots,u_{\fm -n}): B^+_r \rightarrow \mathbb{R}^{\fm-n}$, in the form
\begin{align}
    (x^1,\cdots,x^{n-1},x^{\fm})\mapsto\bigl(x^1,\cdots,x^{n-1},u_1,\cdots,u_{\fm -n},x^{\fm}\bigr).
\end{align}
Setting $y^1=x^1,\cdots,y^{n-1}=x^{n-1},y^n=x^\fm$ as in \eqref{eq-y}, we may rewrite this graph mapping as
\begin{align}\label{eq-u-y}
    (y^1,\cdots,y^n)\mapsto \bigl(y^1,\cdots,y^n,u_1,\cdots,u_{\fm-n}\bigr).
\end{align}
For the remainder of this section, we adopt the coordinate system $y$ in place of $x$.

Let $\mathbf{e}_{i}$ denote the $i$-th standard basis vector of $H^+ \subseteq \mathbb{R}^\fm$. The corresponding $i$-th tangent vector to the graph of $\mathbf{u}$ is given by
\begin{align}
    (\mathbf{e}_{i}, \mathbf{u}_{,i}) := \bigl(\mathbf{e}_{i}, (u_1)_i, \cdots, (u_{\fm-n})_i\bigr).
\end{align}
The graph of $\mathbf{u}$ admits $\fm -n$ distinct normal vectors, which take the form
\begin{align}\label{eq-nu}
    \nu_s = \bigl(D_y u_s, -\mathbf{e}_{s}\bigr) = \bigl((u_s)_1, \cdots, (u_s)_{n}, -\mathbf{e}_{s}\bigr),
\end{align}
for $s=1, \cdots, \fm -n$.

We define the graph mapping
\begin{align}\label{eq-F}
    \mathbf{F}(y) = (y, \mathbf{u}(y)).
\end{align}
For convenience, we extend the coordinate $y$ on $H^+$ to a system of Euclidean coordinates in $\mathbb R^\fm$, where
\begin{align}\label{eq-Y}
    y^{n+1} = x^n, \quad y^{n+2} = x^{n+1}, \quad \cdots, \quad y^\fm = x^{\fm-1}.
\end{align}
Then, for indices satisfying $1\leq i\leq n$ and $n+1\leq l\leq \fm$, we have
\begin{align}
    \mathbf{F}_i = y^i, \quad \mathbf{F}_l = \mathbf{u}_l.
\end{align}

Let $M$ denote the graph of $\mathbf{u}$ (i.e., $M = \mathbf{F}(G_r)$). The mean curvature of $M$ with respect to the unit normal vector $\nu_s/\|\nu_s\|$ is given by
\begin{align}\label{eq-Hj}
    (n-1)H_s = g^{ij}_M \left(\mathbf{F}_{, ij}, \nu_s/\|\nu_s\|\right)_{g_{\mathbb{H}^{\fm}}} = \left(\Delta_{g_M}\mathbf{F}, \nu_s/\|\nu_s\|\right)_{g_{\mathbb{H}^{\fm}}}.
\end{align}
For $1\leq i,j\leq n,$ we have
\begin{align}
    \mathbf{F}_{,ij} &= \nabla_{\mathbf{F}_* \frac{\partial}{\partial y^i}}\left(\mathbf{F}_* \frac{\partial}{\partial y^j}\right)\\
    &=\nabla_{\frac{\partial \mathbf{F}_l}{\partial y^i}\frac{\partial}{\partial y^l}}\left(\frac{\partial \mathbf{F}_k}{\partial y^j}\frac{\partial}{\partial y^k}\right) \\
    &=\frac{\partial^2 \mathbf{F}_k}{\partial y^i \partial y^j}\cdot\frac{\partial}{\partial y^k} +\frac{\partial \mathbf{F}_l}{\partial y^i}\frac{\partial \mathbf{F}_k}{\partial y^j} 
    \cdot \nabla_{\frac{\partial}{\partial y^l}}\frac{\partial}{\partial y^k},
\end{align}
where $\nabla$ denotes the Levi-Civita connection of $\mathbb{H}^{\fm}$. 
Contracting with the inverse metric $g_M^{ij}$ yields
\begin{align}
    \Delta_{g_M} \mathbf{F} = g_M^{ij}\left(\frac{\partial^2 \mathbf{F}_k}{\partial y^i \partial y^j}\cdot\frac{\partial}{\partial y^k} +\frac{\partial \mathbf{F}_l}{\partial y^i}\frac{\partial \mathbf{F}_k}{\partial y^j} 
    \cdot \nabla_{\frac{\partial}{\partial y^l}}\frac{\partial}{\partial y^k} \right).
\end{align}
Recall that $y^{n}$ is the height coordinate in the hyperbolic space; the connection coefficients satisfy the explicit formula
\begin{align}
    \nabla_{\frac{\partial}{\partial y^l}}\frac{\partial}{\partial y^k} = -\frac{1}{y^{n}}\left(\delta_{ln} 
    \frac{\partial}{\partial y^k} +\delta_{k n} 
    \frac{\partial}{\partial y^l}-\delta_{l k} 
    \frac{\partial}{\partial y^{n}}\right).
\end{align}
We compute the quadratic term in the Laplacian expansion. For simplicity, we suppress the Einstein summation convention in this calculation, and obtain
\begin{equation}\label{eq-DXDX}
g_M^{ij}\frac{\partial \mathbf{F}_l}{\partial y^i}\frac{\partial \mathbf{F}_k}{\partial y^j} 
    \cdot \nabla_{\frac{\partial}{\partial y^l}}\frac{\partial}{\partial y^k} = \begin{cases}
    -\frac{1}{y^{n}}g_M^{l k} \left(\delta_{ln}\frac{\partial}{\partial y^{k}} +\delta_{kn}\frac{\partial}{\partial y^{l}}-\delta_{lk}\frac{\partial}{\partial y^n}\right), & 1\leq l,k\leq n;\\
    -\frac{1}{y^{n}}\frac{\partial u_{l-n}}{\partial y^i} 
    \cdot\delta_{kn} g_M^{i k}\frac{\partial}{\partial y^l}, &n+1\leq l\leq \fm \text{ and } 1\leq k \leq n;\\
    \frac{1}{y^{n}}\frac{\partial u_{l-n}}{\partial y^i}\frac{\partial u_{k-n}}{\partial y^j}\cdot\delta_{l k}g_M^{ij}\frac{\partial}{\partial y^{n}}, &n+1\leq l,k \leq \fm.
    \end{cases}
\end{equation}
Combining \eqref{eq-nu}, \eqref{eq-Hj}, and \eqref{eq-DXDX}, together with the metric representation
\begin{align}
    g_{ij}^M = (y^n)^{-2}\left(\delta_{ij} +\sum_{l=1}^{\fm -n} \frac{\partial u_l}{\partial y^i}\frac{\partial u_l}{\partial y^j}\right),
\end{align}
the vanishing mean curvature condition $H_s=0$ implies the following system of PDEs. Defining the rescaled metric
$g_{ij} = (y^n)^{2}g_{ij}^M,$
which coincides with the metric defined in \eqref{eq-gij}, we have
\begin{align}\label{eq-main}
    g^{ij}\frac{\partial^2 u_s}{\partial y^i \partial y^j} - \frac{n}{y^{n}} \frac{\partial u_s}{\partial y^{n}} =0,
\end{align}
for $s=1, \cdots, \fm-n$. This is the primary system of minimal surface equations investigated in the present paper. In summary, we have established Lemma \ref{lem-Main} as well as the following lemma.

\begin{lemma}
If the graph of $\mathbf{u}\in C^1$ in the form given by \eqref{eq-u-y}, defined over a domain $\Omega \subseteq H^+$, is an $n$-dimensional minimal submanifold of $\mathbb{H}^\fm$ (i.e., all mean curvatures vanish), then the PDE system \eqref{eq-main} holds. 
\end{lemma}

\medskip
We may also derive the system \eqref{eq-main} via the method of variations.
Consider an arbitrary smooth map defined on $G_r$, given in the form
\begin{align}
    \Phi: y \mapsto (\phi(y), \eta(y)),
\end{align}
where $\phi = (\phi_1, \cdots, \phi_n)$ and $\eta = (\eta_1, \cdots, \eta_{\fm -n})$. For the graph mapping $\mathbf{F}$ defined in \eqref{eq-F}, we write the induced metric on the graph of $\mathbf{u}$ as
\begin{align}\label{eq-h}
h_{ij}=g_{ij}/(y^n)^2,  
\end{align}
where $(g_{ij})$ is given by \eqref{eq-gij}.
Let $g = \det(g_{ij})$ and $h = \det(h_{ij})$ denote the determinants of the respective metric tensors. Then, for any compact domain $\Omega \subseteq G_r$, the area functional of the graph is given by
\begin{align}
    A(\mathbf{F}): =\mathrm{Area}(\mathbf{F}|_{\Omega})=\int_\Omega\sqrt{h} \,dy=\int_\Omega \frac{1}{(y^n)^n}\sqrt{g}\, dy.
\end{align}
For simplicity, we suppress the volume element $dy$ in the calculation and compute the first variation of the area functional:
\begin{equation}\label{area-F}
\begin{aligned}
\left.\frac{d}{dt}\right|_{t=0}A(\mathbf{F}+t \Phi)
&=\int_{\Omega}\left.\frac{d}{dt}\right|_{t=0}\frac{1}{(y^n+t\phi_n)^n}\sqrt{\det\left(\left((\mathbf{F}+t\Phi)_{y^i},(\mathbf{F}+t\Phi)_{y^j}\right)_{\mathbb{R}^{\fm}}\right)}\\
&=-\int_\Omega \frac{n}{(y^n)^{n+1}}\phi_n\sqrt{g}+\int_\Omega \frac{1}{(y^n)^{n}}\frac{1}{2\sqrt{g}}\left.\frac{d}{dt}\right|_{t=0}\det\left(\left( (\mathbf{F}+t\Phi)_{y^i},(\mathbf{F}+t\Phi)_{y^j}\right)_{\mathbb{R}^{\fm}}\right)\\
&=-\int_\Omega \frac{n}{(y^n)^{n+1}}\phi_n\sqrt{g}+\int_\Omega \frac{1}{(y^n)^{n}}\frac{1}{2\sqrt{g}}\left(\left.\frac{d}{dt}\right|_{t=0}\left( (\mathbf{F}+t\Phi)_{y^i},(\mathbf{F}+t\Phi)_{y^j}\right)_{\mathbb{R}^{\fm}}\right) g g^{ij}\\
&=-\int_\Omega \frac{n}{(y^n)^{n+1}}\phi_n\sqrt{g}+\int_\Omega \frac{1}{(y^n)^n}\sqrt{g}g^{ij}\left( \mathbf{F}_{y^j},\Phi_{y^i}\right)_{\mathbb{R}^{\fm}}.
\end{aligned}
\end{equation}
Next, we select suitable test functions. Let $\zeta\in C^{\infty}_c(\Omega)$ and set
\begin{align}
    \zeta_i =\frac{\partial \zeta}{\partial y^i}.
\end{align}
We denote by $\mathbf{e}_i$ the standard basis vector of $\mathbb{R}^\fm$ whose $i$-th component is $1$ and all other components are $0$.
\begin{itemize}
    \item Take 
    \begin{align}
        \phi(y)=\zeta(y) \mathbf{e}_\alpha,\quad \eta(y)=\mathbf{0},
    \end{align}
    for $\alpha=1,\cdots,n-1$. Substituting this into the first variation formula yields
    \begin{align}
    \int_\Omega \frac{1}{(y^n)^n}\zeta_i\delta_{j\alpha}\sqrt{g}\,g^{ij}=0.    
    \end{align}
    Hence, we deduce the identity
    \begin{equation}\label{choose-phi-alpha}
    \left(\frac{1}{(y^n)^n}\sqrt{g}\,g^{i\alpha}\right)_{,i}=0\quad\text{for }  \alpha=1,\cdots,n-1.
    \end{equation}

    \item Take
    \begin{align}
        \phi(y)=\zeta(y)\mathbf{e}_n,\quad \eta(y)=\mathbf{0}.
    \end{align}
    Substituting this test function into the variation formula gives
    \begin{align}
       -\int_\Omega \frac{n}{(y^n)^{n+1}}\zeta\sqrt{g}+\int_\Omega \frac{1}{(y^n)^n}\zeta_i\delta_{jn}\sqrt{g}\,g^{ij}=0. 
    \end{align}
    Hence, we obtain the divergence identity
    \begin{equation}\label{choose-phi-n}
    \left(\frac{1}{(y^n)^n}\sqrt{g}\,g^{in}\right)_{,i}=-\frac{n}{(y^n)^{n+1}}\sqrt{g}.
    \end{equation}

    \item Take 
    \begin{align}
        \phi(y)=\mathbf{0},\quad \eta(y)=\zeta(y) \mathbf{e}_s,
    \end{align}
    for $s=1,\cdots,\fm-n$. This leads to the integral identity
    \begin{align}\label{eq-choose-phi-0}
        \int_\Omega \frac{1}{(y^n)^n}\sqrt{g}\,g^{ij}(u_s)_{,j}\zeta_{,i}=0.
    \end{align}
\end{itemize}

Combining the identities \eqref{choose-phi-alpha}, \eqref{choose-phi-n}, and \eqref{eq-choose-phi-0}, we compute the divergence of the relevant expression:
\begin{equation}\label{choose-eta}
    \begin{aligned}
   0&=\left(\frac{1}{(y^n)^n}\sqrt{g}\,g^{ij}(u_s)_{,j}\right)_{,i} \\
   &=\left(\frac{1}{(y^n)^n}\sqrt{g}\,g^{ij}\right)_{,i}(u_s)_{,j}+\frac{1}{(y^n)^n}\sqrt{g}\,g^{ij}(u_s)_{,ij}\\
   &=-\frac{n}{(y^n)^{n+1}}\sqrt{g}\,(u_s)_{,n}+\frac{1}{(y^n)^n}\sqrt{g}\,g^{ij}(u_s)_{,ij}.
   \end{aligned}
\end{equation}
Dividing both sides by $\frac{1}{(y^n)^n}\sqrt{g}$ (which is non-vanishing), we conclude that
\begin{equation}\label{main-equation}
    g^{ij}(u_s)_{,ij}-\frac{n}{y^n}(u_s)_{,n}=0,\quad\text{for }s=1,\cdots,\fm -n,
\end{equation}
which coincides with the minimal surface system \eqref{eq-main}.

\section{Higher Order Regularity Theorems}\label{sec-HR}
We investigate the minimal surface system \eqref{eq-main}, where the vector-valued function $\mathbf{u}= (u_1, \cdots, u_{\fm-n})$ is defined on the domain $B^+_r$ specified by \eqref{eq-HR+-0}. The matrix $(g^{ij})$ appearing in the system denotes the inverse of the metric tensor
\begin{align}
    g_{ij} = \delta_{ij} +\sum_{l=1}^{\fm -n} \frac{\partial u_l}{\partial y^i}\frac{\partial u_l}{\partial y^j}.
\end{align}

Theorem \ref{thm-local-C1a} guarantees that the solution $\mathbf{u}$ is of class $C^{1,\alpha}$ up to the boundary. We introduce the tangential domain
\begin{align}
    B_r' = \{y' =(y^1,\cdots, y^{n-1}): |y'| <r\}.
\end{align}

By virtue of the coordinate transformations \eqref{eq-Q-coor}--\eqref{eq-H-coor}, we may impose the following normalized boundary conditions:
\begin{align}\label{eq-varphi-0}
    &\mathbf{u} (y', 0) = \varphi(y'), \quad \text{for } y'\in B_r',\\
    &\mathbf{u} (\mathbf{0}, 0) = \varphi(\mathbf{0}) = \mathbf{0},\quad D\varphi ({0}) = \mathbf{0},\label{eq-varphi-01}
\end{align}
where $\varphi$ is the $(\fm-n)$-valued local defining function of the boundary manifold $\Gamma$, and $\varphi\in C^{1,\alpha}(B_r')$.

\begin{lemma}\label{lem-c1=0}
    Let $\alpha \in (0,1], r>0$ be given constants, and let $\mathbf{u} \in C^{1,\alpha}(\overline{B^+_r}) \cap C^2(B^+_r)$ be a solution to the minimal surface system \eqref{eq-main} in $B^+_r$, satisfying the boundary conditions \eqref{eq-varphi-0}--\eqref{eq-varphi-01}. Suppose further that the closure of the graph of $\mathbf{u}$, denoted by $\mathcal{C}(\operatorname{graph}(\mathbf{u}))\subseteq \mathbb R^\fm$, admits a vertical tangent plane at every point $P$ of the form
    \begin{align}
    P=(y', 0) \in \mathcal{C}(\operatorname{graph}(\mathbf{u})),
    \end{align}
    with respect to the Euclidean metric of $\mathbb{R}^\fm$. Then the following asymptotic expansion holds:
    \begin{align}
        \mathbf{u}(y)= \varphi(y') + O\bigl((y^n)^{1+\alpha}\bigr),
    \end{align}
    where the remainder term $O\bigl((y^n)^{1+\alpha}\bigr)$ is itself a $(\fm-n)$-valued $C^{1,\alpha}$ function.
\end{lemma}

\begin{proof}
    By the Taylor expansion theorem for vector-valued functions, we have the expansion
    \begin{align}
        \mathbf{u}(y) = \varphi(y') + \mathbf{c}_1 (y')y^n + O\bigl((y^n)^{1+\alpha}\bigr),
    \end{align}
    where $\mathbf{c}_1$ is an $(\fm-n)$-valued function on $B_r'$. Substituting this expansion into the definition of the metric component $g_{nn}$, we obtain
    \begin{align}
        g_{nn}(y', 0) = 1 + \bigl|\mathbf{c}_1(y')\bigr|^2.
    \end{align}
    The vertical tangent plane assumption implies that $g_{nn}(y', 0)=1$, which immediately yields $\mathbf{c}_1=\mathbf{0}$. This completes the proof.
\end{proof}

By virtue of Lemma \ref{lem-c1=0}, we may assume the following asymptotic expansion for the metric tensor components:
\begin{align}
    g_{ij} = \delta_{ij} +\sum_{l=1}^{\fm -n} \frac{\partial \varphi_l}{\partial y^i}\frac{\partial \varphi_l}{\partial y^j} +O\bigl((y^n)^\alpha\bigr),
\end{align}
where we adopt the convention that
\begin{align}
    \frac{\partial \varphi_l}{\partial y^n} = 0.
\end{align}

We deduce that for $\beta,\gamma=1,\cdots, n-1$, the tangential metric inverse $(g^{\beta\gamma})$ is the inverse of the matrix
\begin{align}\label{eq-g-inv1}
g_{\beta\gamma}=\delta_{\beta\gamma} +\sum_{l=1}^{\fm -n} \frac{\partial \varphi_l}{\partial y^\beta}\frac{\partial \varphi_l}{\partial y^\gamma} +O\bigl((y^n)^\alpha\bigr),
\end{align}
and the remaining metric inverse components satisfy
\begin{align}\label{eq-g-inv11}
    g^{nn} = 1+O\bigl((y^n)^\alpha\bigr),\quad g^{\beta n} = O\bigl((y^n)^\alpha\bigr).
\end{align}
Here, we note that $\alpha=1$ in the case where $\mathbf{u} \in C^{1,1}(\overline{B^+_r})$.

From the minimal surface system \eqref{eq-main}, the $(\fm-n)$-valued function defined by
\begin{align}
    \mathbf{v} =\mathbf{u} -\varphi
\end{align}
satisfies, for each $s=1,\cdots, \fm -n$, the PDE
\begin{align}\label{eq-g-inv3}
    (y^n)^2 g^{ij}(v_{s})_{,ij} - n y^n(v_s)_{,n} = -(y^n)^2 g^{ij}(\varphi_{s})_{,ij}.
\end{align}
This equation is uniformly elliptic after rescaling in a neighborhood of any point $P \in B^+_{r/2}$. Precisely, under the rescaled coordinate transformation
\begin{align}\label{eq-z}
    z = (y^n(P))^{-1}(x - x(P)),
\end{align}
equation \eqref{eq-g-inv3} is uniformly elliptic on the rescaled domain
\begin{align}
    \left\{z\in\mathbb{R}^n: |z|<\frac{1}{2}\right\}.
\end{align}

Using the auxiliary metric $h$ defined in \eqref{eq-h}, we apply the Schauder estimates and the maximum principle to establish the following lemma.
\begin{lemma}[Tangential Smoothness]\label{lem-tan-smooth}
    Let $\alpha\in (0,1], R>0$ be given constants, and let $\mathbf{u} \in C^{1,\alpha}(\overline{B^+_R})\cap C^2(B^+_R)$ be a solution to the minimal surface system \eqref{eq-main} in $B^+_R$. Suppose further that there exists a smooth function $\varphi \in C^{\infty}(B_R')$ such that
    \begin{align}\label{eq-varphi-1+alpha}
        \mathbf{u} -\varphi = O\bigl((y^n)^{1+\alpha}\bigr).
    \end{align}
    Then for any $r\in (0, R/2)$, any $P\in B_{r}^+$, and any integer $l\in \mathbb N$, there exists a positive constant $C_l>0$ independent of $P$ such that
    \begin{align}\label{eq-tang-smooth}
        \|D_{y'}^l (\mathbf{u} -\varphi)\|_{C^{2, \alpha}_h\bigl(B_h(P, \tfrac{1}{2})\bigr)}\leq C_l(y^n)^{2}.
    \end{align}
\end{lemma}

\begin{proof}
    By means of the rescaled coordinate transformation \eqref{eq-z}, combined with the PDE \eqref{eq-g-inv3}, the asymptotic condition \eqref{eq-varphi-1+alpha}, and the Schauder estimates for elliptic equations, we immediately obtain the a priori estimate
    \begin{align}
          \|\mathbf{u} -\varphi\|_{C^{2, \alpha}_h\bigl(B_h(P, \tfrac{1}{2})\bigr)}\leq C (y^n)^{1+\alpha}.
    \end{align}

    We denote by $D_{y'}^l$ an arbitrary $l$-th order tangential differential operator, which takes the form $D_{y'}^\beta$ where $\beta$ is a multi-index satisfying $|\beta|=l$. Applying the tangential differential operator $D_{y'}^l$ to both sides of \eqref{eq-g-inv3}, we derive the differentiated PDE
    \begin{align}\label{eq-g-inv4}
         g^{ij}(D_{y'}^l v_s)_{,ij} - n  \frac{(D_{y'}^l v_s)_{,n}}{y^n} = - D_{y'}^l\bigl(g^{ij}(\varphi_{s})_{,ij}\bigr) -\sum_{m=0}^{l-1}\binom{l}{m} \bigl(D_{y'}^{l-m} g^{ij}\bigr)(D_{y'}^m v_s)_{,ij},
    \end{align}
    where for any $r\in (0, R/2)$ and any $P\in B_{r}^+$, the $C^\alpha_h\bigl(B_h(P,\frac{1}{2})\bigr)$-norm of the right-hand side is bounded by a constant depending only on
    \begin{align}
    \fm, \, n, \, l, \, \|\mathbf{u}\|_{C^{1,\alpha}(B^+_R)}, \, \|\varphi\|_{C^{l+2,\alpha}(B_R')}
    \end{align}
    and the constants
    \begin{align}\label{eq-C-0}
        C_0, \, C_1, \, \cdots, \, C_{l-1}
    \end{align}
    appearing in the estimate \eqref{eq-tang-smooth} for lower orders. Note that the constants in \eqref{eq-C-0} are redundant when $l=0$.

    For any fixed $\delta\in (0, r)$ and any $y_0' \in B_{r-\delta}'$, we introduce the barrier function
    \begin{align}
        M(y', y^n) = a|y'-y_0'|^2 + b(y^n)^2,
    \end{align}
    where $a,b>0$ are to be determined. Following the argument in the proof of Theorem 3.1 in \cite{HanJiang2023}, we apply the maximum principle for elliptic PDEs to establish the $L^\infty$-estimate
    \begin{align}
        \| D_{y'}^l (\mathbf{u} -\varphi) \|_{L^\infty(B^+_{r-\delta})}\leq C (y^n)^{2}.
    \end{align}
    Finally, by shrinking the radius $r$ if necessary, we complete the proof by invoking the rescaling argument once again.
\end{proof}
\begin{remark}\label{rem-tan}
    The estimate \eqref{eq-tang-smooth} can be improved to be, for any $q\in \mathbb{N},$
    \begin{align}\label{eq-tang-smooth-1}
        \|D_{y'}^l (\mathbf u -\varphi)\|_{C^{q, \alpha}_h(B_h(P, \frac{1}{2}))}\leq C_{q,l}(y^n)^{2}.
    \end{align} 
\end{remark}
Meanwhile, the minimal surface system \eqref{eq-main} can be recast in the following form for each $s=1,\cdots, \fm -n$:
\begin{align}\label{eq-g-inv2}
    (u_s)_{,nn} - \frac{n}{y^n}(u_s)_{,n} = (u_s)_{,nn}- g^{ij}(u_{s})_{,ij},
\end{align}
where the right-hand side contains only second-order tangential derivatives and mixed derivatives of $u_s$, provided that we neglect the derivatives of $\mathbf{u}$ appearing in the metric coefficients $g^{ij}$.

For any $(\fm-n)$-valued function $\mathbf{w} = (w_1,\cdots, w_{\fm-n})$ defined on $B^+_R$, we define the differential operator $\mathcal{Q}[\mathbf{w}]$ by
\begin{align}\label{eq-Q}
    \mathcal{Q}[\mathbf{w}] :=g^{ij}[\mathbf{w}]\frac{\partial^2 w_s}{\partial y^i \partial y^j} - \frac{n}{y^{n}} \frac{\partial w_s}{\partial y^{n}},
\end{align}
where $g^{ij}[\mathbf{w}]$ denotes the inverse of the metric tensor
\begin{align}
    g_{ij}[\mathbf{w}] =\delta_{ij} +\sum_{l=1}^{\fm -n} \frac{\partial w_l}{\partial y^i}\frac{\partial w_l}{\partial y^j}.
\end{align}
By means of the metric inverse estimates \eqref{eq-g-inv1}--\eqref{eq-g-inv11}, we establish the following lemma on formal series solutions.

\begin{lemma}[Formal Series Solution]\label{lem-Form-Comp}
    Let $\varphi \in C^\infty(B_r')$ be a smooth function. Then there exist smooth $(\fm-n)$-valued functions $\mathbf{c}_i$ and $\mathbf{c}_{i,j}$ on $B_r'$ (with $(i,j)\neq (n+1,0)$) that depend on $\varphi$ and the first global coefficient $\mathbf{c}_{n+1,0}$ such that, for any given smooth $(\fm-n)$-valued function $\mathbf{c}_{n+1,0}$ on $B_r'$, there exists a unique formal series of the form
    \begin{align}\label{eq-expan}
        \mathbf{u} = \varphi +\sum_{i=2}^{n}\mathbf{c}_i(y')(y^n)^i +\sum_{i=n+1}^\infty \sum_{j=0}^{\left\lfloor \frac{i-1}{n}\right\rfloor} \mathbf{c}_{i,j}(y')(y^n)^i\bigl(\log (y^n)\bigr)^j \quad \text{in } B^+_r,
    \end{align}
    which formally solves the minimal surface system \eqref{eq-main}. Here  $\mathbf{c}_i(y')=\mathbf{0}$ whenever $i\leq n$ is an odd integer.

    In addition, we define the partial sum $\mathbf{T}_k$ by
    \begin{align}\label{eq-uk}
        \mathbf{T}_k = \varphi +\sum_{\substack{i=2 \\ i\ \text{even}}}^{2\left\lfloor \frac{n}{2}\right\rfloor}\mathbf{c}_i(y') (y^n)^i +\sum_{i=n+1}^k\sum_{j=0}^{\left\lfloor \frac{i-1}{n}\right\rfloor} \mathbf{c}_{i,j}(y') (y^n)^i\bigl(\log (y^n)\bigr)^j.
    \end{align}
    This partial sum satisfies the asymptotic estimate
    \begin{align}
        \mathcal{Q}(\mathbf{T}_k)=O\Bigl((y^n)^{k}\bigl(\log (y^n)\bigr)^{\left\lfloor \frac{k-1}{n}\right\rfloor}\Bigr).
    \end{align}
\end{lemma}

\begin{remark}
    When $n$ is even, we have the identity
    \begin{align}
        2\left\lfloor \frac{n}{2}\right\rfloor =n,
    \end{align}
    in which case all logarithmic terms in the series expansion \eqref{eq-expan} vanish identically.
\end{remark}

\begin{proof}[Proof of Lemma \ref{lem-Form-Comp}]
    By \eqref{eq-derivative}, we recast equation \eqref{eq-g-inv2} in the following equivalent form:
    \begin{align}\label{eq-u-form}
        (y^n)^2\mathbf{u}_{,nn} - n y^n\mathbf{u}_{,n} =(y^n)^2\bigl( \mathbf{u}_{,nn}- g^{ij}\mathbf{u}_{,ij}\bigr).
    \end{align}

    We initialize the coefficient recursion by setting $\mathbf{c}_0=\varphi$, and we proceed by induction: assume that we have formally determined the partial sum $\mathbf{T}_{i-1}$ for some integer $i\geq 2$.

    Consider a general term $\mathbf{c}_{i,j} (y^n)^i\bigl(\log (y^n)\bigr)^j$ appearing in the series expansion \eqref{eq-expan} of $\mathbf{u}$. Substituting this term into the left-hand side of \eqref{eq-u-form}, we compute its leading contribution, which is given by
    \begin{align}\label{eq-c-domi}
        i\bigl(i-(n+1)\bigr)\mathbf{c}_{i,j} (y^n)^i\bigl(\log (y^n)\bigr)^j.
    \end{align}

    Next, we expand the right-hand side of \eqref{eq-u-form} as a power series (with logarithmic terms) and isolate the terms involving $\mathbf{c}_{i,j}$. A key observation is that all such terms carry a factor of $(y^n)^{i+1}$ (or higher order, if logarithmic factors are disregarded). 

    Therefore, whenever 
    \begin{align}
        i\neq n+1,
    \end{align}
    we can uniquely solve for the coefficient $\mathbf{c}_{i,j}$ by equating the coefficients of the term $(y^n)^i\bigl(\log (y^n)\bigr)^j$ on both sides of \eqref{eq-u-form}. The resulting expression for $\mathbf{c}_{i,j}$ depends only on the earlier coefficients $\mathbf{c}_{i', j'}$ with $i'<i$.

For $k\leq n$, we first establish that $\mathbf{T}_k$ contains only even powers of $y^n$ via mathematical induction. We initialize the base case with $\mathbf{T}_2 = \varphi + \mathbf{c}_2(y')(y^n)^2$, which trivially consists of even powers of $y^n$ alone. Assume the inductive hypothesis holds: for any even integer  $k\in [2, n-1] $, the partial sum $\mathbf{T}_k$ contains only even powers of $y^n$.

We substitute $\mathbf{u}=\mathbf{T}_k$ into the right-hand side of \eqref{eq-u-form} and analyze the metric coefficient expansions. A key observation is that for $\beta,\gamma=1,\cdots, n-1$, the expansion of $g^{\beta n}[\mathbf{T}_k]$ comprises only odd powers of $y^n$, while the expansion of $g^{\beta \gamma}[\mathbf{T}_k]$ comprises only even powers of $y^n$. Combining these parity properties, we conclude that the entire right-hand side of \eqref{eq-u-form} is a sum of even powers of $y^n$.

By substituting $\mathbf{u}=\mathbf{T}_{k+1}$ into \eqref{eq-u-form} and invoking the earlier analysis of the leading term formula \eqref{eq-c-domi}, we uniquely deduce that $\mathbf{c}_{k+1}=\mathbf{0}$. This verifies the inductive step, and the claim follows by the principle of mathematical induction.

Next, we consider the critical case $i = n+1$. For $j\geq 1$, the leading term involving $\mathbf{c}_{i,j}$ on the left-hand side of \eqref{eq-u-form} simplifies to
\begin{align}
    (2i-1-n)j\mathbf{c}_{i,j}(y^n)^i\bigl(\log (y^n)\bigr)^{j-1} = (n+1)j\mathbf{c}_{n+1,j} (y^n)^{n+1}\bigl(\log (y^n)\bigr)^{j-1}.
\end{align}
Following the identical coefficient-matching procedure as above, we can uniquely solve for the coefficients $\mathbf{c}_{n+1,j}$ for all $j\geq 1$.

Finally, we note that the coefficient $\mathbf{c}_{n+1,0}$ cannot be determined by this recursive process—it serves as a free parameter. If we prescribe a smooth $(\fm-n)$-valued function $\mathbf{c}_{n+1,0}$ on $B_r'$, then the entire recursive scheme yields a unique formal series solution of the form \eqref{eq-expan}.
\end{proof}

For notational simplicity, we set $t = y^n$ throughout the subsequent discussion.

\begin{definition}\label{def-exp}
A function $\mathbf{w}$ (which may be single-valued or $(\fm-n)$-valued) is said to admit an expansion of order $t^k$ in $B^+_r$ if there exist smooth functions $\mathbf{c}_{i,j}(y')$ on $B_r'$, positive integers $N_i\in \mathbb{N}$, and a constant $\epsilon>0$ such that the decomposition
\begin{align}\label{eq-Taylor}
    \mathbf{w} = \mathbf{T}_k +\mathbf{R}_k := \sum_{i=0}^k \sum_{j=0}^{N_i} \mathbf{c}_{i,j} t^i (\log t)^j+\mathbf{R}_k
\end{align}
holds in $\overline{B^+_r}$. Here, the remainder term $\mathbf{R}_k$ satisfies two key conditions:
\begin{align}\label{eq-Rk-Cke}
 \mathbf{R}_k \in O(t^{k+\epsilon})\cap C^{k,\epsilon}(\overline{B^+_r}),
\end{align}
and for all nonnegative integers $p, q \in \mathbb{N}$, the tangential derivative estimate
\begin{align}\label{eq-Rk-Tan}
    \|D_{y'}^p \mathbf{R}_k \|_{C^{q,\alpha}_h(G_r)}\leq C_{p,q,k}\,t^{k+\epsilon}
\end{align}
is valid, where $C_{p,q,k}$ denotes a constant independent of $t$ and $y'$.
\end{definition}

We refer to the decomposition \eqref{eq-Taylor} as a Taylor expansion with logarithmic terms of $\mathbf{w}$.

By virtue of Lemma \ref{lem-tan-smooth} and Remark \ref{rem-tan}, we immediately conclude that the solution $\mathbf{u}$ admits an expansion of order $t$.

To derive the full Taylor expansion of $\mathbf{u}$ with logarithmic terms, we proceed by iteratively applying the following lemma.
\begin{lemma}\label{lem-Taylor-u}
    Assume the same hypotheses as in Lemma \ref{lem-tan-smooth}. For any integer $k\geq 1$, if $\mathbf{u}$ admits an expansion of order $t^k$ in $B^+_R$, then $\mathbf{u}$ admits an expansion of order $t^{k+1}$ in $B^+_r$ for every $r\in (0, R)$.
\end{lemma}

\begin{proof}
    We split the proof into two key steps.

    First, we prove that if $\mathbf{u}$ admits an expansion of order $t^k$ in $B^+_r$, then the nonlinear forcing term defined by
    \begin{align}\label{eq-rem-1}
        F_s(y, \mathbf{u}):=t^2(u_s)_{,tt}- t^2 g^{ij}(u_{s})_{,ij}
    \end{align}
    (cf. equation \eqref{eq-u-form}) admits an expansion of order $t^{k+1}$ in $B^+_r$.

    By the inductive hypothesis, $\mathbf{u}$ satisfies the decomposition \eqref{eq-Taylor} with remainder $\mathbf{R}_k$ of order $t^{k+\epsilon}$. Substituting this expansion into \eqref{eq-rem-1}, we observe that the remainder contribution $\mathbf{R}_k$ acquires an additional factor of $t$ after differentiation and multiplication by $t^2$. This claim follows directly from formal differentiation and substitution; for a detailed computation, we refer the reader to the proof of Lemma 5.2 in \cite{JiangShi}.

    Next, we show that if $F_s$ admits an expansion of order $t^{k+1}$, then so does $u_s$. The key tool is the integral representation of the solution to the ODE derived from \eqref{eq-g-inv2}. For $t\in (0, r)$, we have
    \begin{align}
        u_s(y', t) = &\left[ u_s(y', r) r^{-\overline{m}} - \frac{r^{\underline{m} - \overline{m}}}{\overline{m} - \underline{m}} \int_0^r \zeta^{-1 - \underline{m}} F_s(y', \mathbf{u}(y', \zeta)) \, d\zeta \right] t^{\overline{m}}
        \nonumber\\
        &\qquad + \frac{t^{\underline{m}}}{\overline{m} - \underline{m}} \int_0^t \zeta^{-1 - \underline{m}} F_s(y', \mathbf{u}(y', \zeta)) \, d\zeta
        \nonumber\\
        &\qquad + \frac{t^{\overline{m}}}{\overline{m} - \underline{m}} \int_t^r \zeta^{-1 - \overline{m}} F_s(y', \mathbf{u}(y', \zeta)) \, d\zeta,\label{eq-log-int}
    \end{align}
    where the exponents are given by $\overline{m} = n+1$ and $\underline{m} =0$.

    A critical observation is that the $t^{\overline{m}}$ term in the expansion of $F_s(y,\mathbf{u})$ generates the first logarithmic correction term $t^{\overline{m}}\log t$ in $u_s$ via the integral terms in \eqref{eq-log-int}. Note that the regularity exponent $\epsilon$ in the remainder estimate \eqref{eq-Rk-Cke} may decrease when $k+1 \geq \overline{m}$, due to the introduction of logarithmic terms. For full details of this integral analysis, we refer to Section 4 in \cite{HanJiang2023}.
\end{proof}

From Lemma \ref{lem-Taylor-u} and the proof strategy of Theorem 5.3 in \cite{HanJiang2023}, we deduce Theorem \ref{thm-BRT1}. Theorems \ref{thm-Cna} and \ref{thm-BRT2} are finite-regularity analogues of Theorem \ref{thm-BRT1}; their proofs follow the same line of reasoning as that of Theorem 5.2 in \cite{HanJiang2023}.

\section{Geometric Structure of the Expansion Coefficients}\label{sec-coeff}
In the hypersurface case in \cite{HanJiang2023}, the coefficients of the local terms in the expansion of the solution $u$ can be expressed in terms of several geometric quantities, such as the mean curvature, Gaussian curvature, and derivatives of curvature. In the higher-codimension case, the computation becomes more involved, yet we can simplify it by using linear transformations in Euclidean space.

\begin{lemma}\label{lem-c2} Assume the same hypotheses as in Lemma \ref{lem-Form-Comp}. Then the $s$ component of $\mathbf c_{2}$ is
\begin{align}\label{eq-c-s2}
    c_{s,2}(y') = \frac{1}{2}\sqrt{1+|D\varphi_s(y')|^2} (H(y'), \nu_s(y')),
\end{align}
where $H$ is the mean curvature vector of $\Gamma$ at 
\begin{align}
 (y', 0, \varphi(y')) = (y^1,\cdots, y^{n-1}, 0, \varphi((y^1,\cdots, y^{n-1})),  
\end{align} and $\nu_s$ is a unit normal vector in the form of
\begin{align}
    \nu_s = \frac{(-D\varphi_s, \mathbf e_s)}{\sqrt{1+|D \varphi_s|^2}}.
\end{align}
\end{lemma}
\begin{proof}
    Similarly to \eqref{eq-u-form}, we can rewrite the main equation \eqref{eq-main} as
    \begin{align}\label{eq-main-comp1}
        u_{s, tt} -\frac{n}{t} u_{s,t} = \frac{n}{t}((g^{nn})^{-1}-1)u_{s, t} -(g^{nn})^{-1}\sum_{(i,j)\neq (n,n)}g^{ij}u_{s,ij},
    \end{align}
    where $t=y^n.$ We substitute $\mathbf{u}=\varphi$ into the right-hand side of \eqref{eq-main-comp1} and obtain $c_{s,2}$ on the left-hand side by matching the  $O(1)$ terms. 
    As \begin{align}
        g^{nn}=\frac{\det((g_{\alpha\beta}))}{\det (g)},
    \end{align}
where $(g_{\alpha\beta})$ denotes the left top $(n-1)\times (n-1)$ submatrix of $g,$ we have
\begin{align}
    (g^{nn})^{-1}-1 &= \frac{\det(g) - \det ((g_{\alpha\beta}))}{\det (g)}\\ &= \frac{\sum_{\alpha< n}g_{n\alpha}\det ((g^*)^{n\alpha})+ (g_{nn}-1) \det ((g_{\alpha\beta}))}{\det ((g_{\alpha\beta}))}\\ &= \frac{\sum_{\alpha< n}g_{n\alpha}\det ((g^*)^{n\alpha})+\sum_{s=1}^{\fm-n} u_{s,t}^2 \det ((g_{\alpha\beta}))}{\det ((g_{\alpha\beta}))},\label{eq-g-nn}
\end{align}
where $g^*$ denotes the cofactor matrix of $g.$ Substituting $\mathbf u=\varphi$ and the corresponding $g$ into \eqref{eq-g-nn} yields 0. Thus
\begin{align}\label{eq-c2-local}
    c_{s,2} =\frac{1}{2n-2}\sum_{1\leq \alpha,\beta<n}g^{\alpha\beta} \varphi_{s,\alpha \beta}. 
\end{align}
Recall from Section \ref{sec-PDE} that we have, for $\mathbf u=\varphi$ and $\alpha=1,\cdots, n-1,$
\begin{align}
    \mathbf F_{,\alpha} :=\mathbf{F}_* \frac{\partial}{\partial y^\alpha}= \partial_\alpha + \sum_{\xi=1}^{\fm-n} \varphi_{\xi, \alpha}\partial_{\xi}
\end{align}
and the second fundamental form is
\begin{align}\label{eq-II}
\mathrm{II}_{\alpha \beta} &= \left(\nabla_{\mathbf{F}_{,\alpha}}\mathbf{F}_{,\beta}\right)^\perp = \left(\sum_{\xi=1}^{\fm-n} \varphi_{\xi, \alpha\beta}\partial_{\xi}\right)^\perp.
\end{align}
Thus
\begin{align}\label{eq-c2-alter}
    c_{s,2} = \frac{1}{2n-2}\left(\sum_{1\leq \alpha,\beta<n} g^{\alpha\beta} \mathrm{II}_{\alpha \beta}, \sqrt{1+|D\varphi_s|^2}\nu_s\right)= \frac{1}{2}\sqrt{1+|D\varphi_s|^2}(H, \nu_s),
\end{align}
which concludes the lemma.
\end{proof}

Let $\operatorname{II}$ denote the second fundamental form of $\Gamma$. For any orthonormal tangent frame $\{e_i\}$, we define
\begin{align}
\operatorname{II}_{ij} = \operatorname{II}(e_i, e_j)
\in \Gamma(N\Gamma).
\end{align}
We introduce an operator $Q: \Gamma(N\Gamma)\to \Gamma(N\Gamma)$ by
\begin{align}
Q(\operatorname{II})(V) = \sum_{i,j} ( V, \operatorname{II}_{ij}) \operatorname{II}_{ij}.
\end{align}
Then for any section $V\in \Gamma(N\Gamma)$, the following holds:
\begin{align}\label{eq-Q-II}
(\Delta_\Gamma V)^\perp = \Delta^\perp V - Q(\operatorname{II})(V).
\end{align}
This identity can be verified using the Weingarten map $A_V$ defined by
\begin{align}
A_V(e_i) = \nabla_{e_i}^\perp V - \nabla_{e_i} V,
\end{align}
which satisfies
\begin{align}
( A_V(e_i), e_j) = ( V, \operatorname{II}_{ij}).
\end{align}

We now proceed to compute $\mathbf{c}_{4,1}$ in the case $n=3$.
\begin{theorem}\label{thm-c41}
Assume the same hypotheses as in Lemma \ref{lem-Form-Comp}. If in addition $n=3$, then the $s$-component of $\mathbf{c}_{4,1}$ is
\begin{align}\label{eq-c-41}
c_{s, 4,1} = -\frac{1}{8}\sqrt{1+|D\varphi_s(y')|^2} \left( (\Delta_\Gamma H)^\perp-2|H|^2 H + 2Q(\operatorname{II})(H), \nu_s \right),
\end{align}
where $H$ and $\operatorname{II}$ denote the mean curvature vector and the second fundamental form of $\Gamma$ at
\begin{align}
(y', 0, \varphi(y')) = \bigl(y^1,\cdots, y^{n-1}, 0, \varphi(y^1,\cdots, y^{n-1})\bigr),
\end{align}
and $\nu_s$ is the unit normal vector given by
\begin{align}
\nu_s = \frac{(-D\varphi_s, \mathbf{e}_s)}{\sqrt{1+|D \varphi_s|^2}}.
\end{align}
\end{theorem}

\begin{proof}
    First, we rewrite \(c_{s,2}\). By \eqref{eq-II},
\begin{align}\label{eq-H-exp}
    (n-1)H = \left(g^{\alpha\beta}\sum_{\xi=1}^{\fm-n} \varphi_{\xi, \alpha\beta}\partial_{\xi}\right)^\perp
    = g^{\alpha\beta}\sum_{\xi=1}^{\fm-n} \varphi_{\xi, \alpha\beta}\partial_{\xi}
    - g^{\alpha\beta} (\varphi_{\alpha\beta}\cdot \varphi_\gamma)g^{\gamma\eta}\left(\partial_\eta+ \sum_{\xi=1}^{\fm-n} \varphi_{\xi, \eta} \partial_{\xi}\right),
\end{align}
where \(\alpha, \beta, \gamma, \eta\) are summed from \(1\) to \(n-1\). By \eqref{eq-c2-local} and \eqref{eq-H-exp},
\begin{align}\label{eq-c2-H}
    c_{s,2} = \frac{1}{2}(H, \partial_s) +\frac{1}{2n-2} g^{\alpha\beta} (\varphi_{\alpha\beta}\cdot \varphi_\gamma)g^{\gamma\eta}\varphi_{s, \eta}.
\end{align}

Next, we compute \(c_{s, 4,1}\) under a special coordinate system. Let \(Q\in \Gamma\) be the origin of the coordinate system \(y\), which satisfies
\begin{align}
    y^{n+1} = \cdots = y^\fm = 0
\end{align}
on the vertical tangent plane \(H^+\) associated with \(Q\). Clearly, the corresponding defining function \(\varphi\) of \(\Gamma\) satisfies, for \(s=1,\dots, \fm-n\),
\begin{align}\label{eq-Dvarphi-0}
    D\varphi_s (0) = 0.
\end{align}

At the point where
\begin{align}
    y' = (y^1,\dots, y^{n-1}) = 0,
\end{align}
substitute \(u = \varphi + \mathbf{c}_2(y') t^2\) into the right-hand side of \eqref{eq-main-comp1} to obtain the \(O(t^2)\) terms as the coefficient of \(t^2\) in the expression
\begin{align}
    n\left(4\sum_{\xi=1}^{\fm-n} c_{\xi, 2}^2t^2\right)2c_{s,2}
    - \left(1+4 \sum_{\xi=1}^{\fm-n} c_{\xi,2}^2 t^2\right)\Delta' \left(\varphi_s +c_{s,2}t^2\right),
\end{align}
where at \(Q\),
\begin{align}
    \Delta' c_{s,2} = \frac{1}{2}(\Delta'H,\partial_s) + (H_{\xi} \cdot \operatorname{II}_{\xi,\gamma\eta}) \operatorname{II}_{s,\gamma\eta}
\end{align}
by \eqref{eq-c2-H}. Then the \(O(t^2)\) terms on the right-hand side of \eqref{eq-main-comp1} are
\begin{align}\label{eq-c-4-RHS}
    |H|^2 H_s - \frac{1}{2}(\Delta'H,\partial_s) - (H_{\xi} \cdot \operatorname{II}_{\xi,\gamma\eta}) \operatorname{II}_{s,\gamma\eta}.
\end{align}
Thus we obtain at \(Q\) that
\begin{align}\label{eq-c41-Q}
    c_{s,4,1} = -\frac{1}{8} \left( (\Delta'H,\partial_s)-2|H|^2 H_s + 2(H_{\xi} \cdot \operatorname{II}_{\xi,\gamma\eta}) \operatorname{II}_{s,\gamma\eta}\right).
\end{align}

Thirdly, we consider the general case. By \eqref{eq-u-exp-0}, under the special coordinate system $y$ satisfying \eqref{eq-Dvarphi-0},
\begin{align}\label{eq-uQ-y}
    \mathbf u_Q(y', t) = \sum_{i=0}^k\sum_{j=0}^{\lfloor \frac{i-1}{n}\rfloor} \mathbf{d}_{i,j}(y')\, t^i (\log t)^j + \tilde{\mathbf R}_k,
\end{align}
in the sense of Definition \ref{def-exp}. Clearly, $\varphi = \mathbf d_{0,0}$, and
\begin{align}\label{eq-uQ-y-2}
    \mathbf u_Q(y', t) = \varphi(y') + \mathbf{d}_{2,0}(y')t^2 + \sum_{i=4}^k\sum_{j=0}^{\lfloor \frac{i-1}{n}\rfloor} \mathbf{d}_{i,j}(y')\, t^i (\log t)^j + \tilde{\mathbf R}_k.
\end{align}

Now fix a point $P\in \Gamma$ near $Q$, and take a special coordinate system $z$ at $P$ such that at $(z', t) = (0, 0)$,
\begin{align}
    D_{z'} \mathbf u_Q = 0.
\end{align}
By \eqref{eq-u-exp-0}, there exist functions $\mathbf c_{i,j}(y')$ such that
\begin{align}\label{eq-uP-exp}
    \mathbf u_P(z', t) = \sum_{i=0}^k\sum_{j=0}^{\lfloor \frac{i-1}{n}\rfloor} \mathbf{c}_{i,j}(z')\, t^i (\log t)^j + \mathbf R_k
\end{align}
holds under the coordinate system $z$. Our goal is to evaluate $\mathbf c_{4,1}$ at $Q$ for $n=3$.

Since $z$ and $y$ differ by a translation and a linear transformation, there exists an $\fm\times \fm$ matrix $B$ such that
\begin{align}
    y = B(z - z(Q)).
\end{align}
In components, for $i=1,\dots,\fm$,
\begin{align}
    y_i = \sum_{l=1}^\fm B_{il}\,(z-z(Q))_l.
\end{align}
We assume that $y$ and $z$ are Euclidean coordinates in $\mathbb{R}^\fm$ with $y^n=z^n=t$, which implies
\begin{align}\label{eq-B-in}
    B_{in} = \delta_{in}.
\end{align}
Using \eqref{eq-B-in} and \eqref{eq-uP-exp}, we obtain on the minimal surface that, for $s>n$,
\begin{align}
y_s &= \sum_{l=1}^\fm B_{sl}\,(z-z(Q))_l \nonumber\\
&= \sum_{l=1}^{n-1} B_{sl}\,(z-z(Q))_l - \sum_{l=1}^\fm B_{sl}\,z_l(Q) + \sum_{\xi=n+1}^\fm B_{s\xi}\,z_\xi \nonumber\\
&= \sum_{l=1}^{n-1} B_{sl}\,(z-z(Q))_l - \sum_{l=1}^\fm B_{sl}\,z_l(Q) \nonumber\\
&\quad + \sum_{\xi=n+1}^\fm B_{s\xi} \left( \sum_{i=0}^k\sum_{j=0}^{\lfloor \frac{i-1}{n}\rfloor} c_{\xi,i,j}(z')\, t^i (\log t)^j + R_{\xi,k} \right). \label{eq-uQ-z}
\end{align}

Fix $z' = z'(Q)$ and vary $t$ to obtain a vertical ray $L^+ \subseteq H_P^+$, defined as in \eqref{eq-H+-0} with $Q$ replaced by $P$. Then the graph map
\begin{equation}
    \bigl(z',\, t,\, \mathbf{u}_P(z', t)\bigr)
\end{equation}
maps $L^+$ onto a curve contained in the minimal surface. We take derivatives with respect to $t$ along $L^+$ and compute
\begin{align}
    \lim_{t\rightarrow 0}\left(\frac{1}{\log t} \frac{\partial^4 y_s}{\partial t^4} \right)
    = \lim_{t\rightarrow 0}\left(\frac{1}{\log t} \frac{\partial^4 u_{Q,s}}{\partial t^4} \right)
\end{align}
in two different ways. Applying \eqref{eq-uQ-z}, we derive
\begin{align}
    \lim_{t\rightarrow 0}\left(\frac{\partial^2 y_s}{\partial t^2} \right)
    &= 2 \sum_{\xi=n+1}^\fm B_{s\xi} \, c_{\xi,2,0}(Q), \label{eq-c2-trans}\\[4pt]
    \lim_{t\rightarrow 0}\left(\frac{1}{\log t} \frac{\partial^4 y_s}{\partial t^4} \right)
    &= 24 \sum_{\xi=n+1}^\fm B_{s\xi} \, c_{\xi,4,1}(Q). \label{eq-c41-trans}
\end{align}

On the other hand, $y_s = u_{Q, s}(y', t) = u_{Q, s}(y'(z), t)$. To apply \eqref{eq-uQ-y}--\eqref{eq-uQ-y-2}, we compute
\begin{align}
    \frac{\partial d_{s, i,j}}{\partial t} &= \sum_{\alpha=1}^{n-1}\sum_{\xi =n+1}^\fm \frac{\partial d_{s, i,j}}{\partial y^\alpha} \frac{\partial y^\alpha}{\partial z^\xi} \frac{\partial z^\xi}{\partial t} = \sum_{\alpha=1}^{n-1}\sum_{\xi =n+1}^\fm \frac{\partial d_{s, i,j}}{\partial y^\alpha} B_{\alpha\xi} \frac{\partial z^\xi}{\partial t}, \\
    \frac{\partial^2 d_{s, i,j}}{\partial t^2} &= \sum_{\alpha,\beta=1}^{n-1}\sum_{\xi,\zeta =n+1}^\fm \frac{\partial^2 d_{s, i,j}}{\partial y^\alpha \partial y^\beta} B_{\alpha\xi} \frac{\partial z^\xi}{\partial t} B_{\beta\zeta} \frac{\partial z^\zeta}{\partial t}+\sum_{\alpha=1}^{n-1}\sum_{\xi =n+1}^\fm \frac{\partial d_{s, i,j}}{\partial y^\alpha} B_{\alpha\xi} \frac{\partial^2 z^\xi}{\partial t^2}.
\end{align}
For $(i,j)=(0,0)$, we have $d_{s,0,0}=\varphi_s$, and at $Q$,
\begin{align}
    \frac{\partial d_{s, 0,0}}{\partial t}=0, \quad \frac{\partial^2 d_{s, 0,0}}{\partial t^2} (y', 0)=0,
\end{align}
since as $t\rightarrow 0$,
\begin{align}
    \frac{d z^\xi}{d t} = 2c_{\xi, 2, 0} t+o(t) \rightarrow 0.
\end{align}
Then at $Q$, by \eqref{eq-uQ-y-2} and \eqref{eq-c2-trans},
\begin{align}
   \lim_{t\rightarrow 0}\left(\frac{\partial^2 y_s}{\partial t^2} \right) = 2 d_{s, 2, 0}(Q)=2 \sum_{\xi=n+1}^\fm B_{s\xi} c_{\xi,2,0}(Q).
\end{align}
Thus the matrix $(B_{s\xi})_{n+1\leq s, \xi\leq \fm}$ maps $\mathbf c_{2,0}(Q)$ to $\mathbf d_{2,0}(Q) = \frac{1}{2}H(Q)$, where $\mathbf c_{2,0}(Q)$ is given by Lemma \ref{lem-c2} with $\varphi$ replaced by $\varphi_P := \mathbf d_{0,0}$ and $y'$ replaced by $z'(Q)$.

Since the fourth-order derivative of $\varphi(y')+ \mathbf d_{2, 0} t^2$ with respect to $t$ along the ray $L^+$ is bounded, we obtain
\begin{align}
    \lim_{t\rightarrow 0}\left(\frac{1}{\log t} \frac{\partial^4 y_s}{\partial t^4} \right) = 24 d_{s,4,1}(Q).
\end{align}
By \eqref{eq-c41-trans}, it follows that
\begin{align}
    d_{s,4,1}(Q) = \sum_{\xi=n+1}^\fm B_{s\xi} c_{\xi,4,1}(Q).
\end{align}
Hence the matrix $(B_{s\xi})_{n+1\leq s, \xi\leq \fm}$ also maps $\mathbf c_{4,1}(Q)$ to
\begin{align}
\mathbf d_{4,1}(Q) =  -\frac{1}{8} \left( \Delta_\Gamma H-2|H|^2 H + 2(H_{\xi} \cdot \operatorname{II}_{\xi,\gamma\eta}) \operatorname{II}_{\cdot,\gamma\eta}\right),
\end{align}
where $H$, $\Delta_\Gamma H$, and $\operatorname{II}$ are evaluated at $Q$. We then apply Lemma \ref{lem-c2}  to complete the proof.
\end{proof}
The formula
\eqref{eq-cs41} follows from Theorem \ref{thm-c41} and  \eqref{eq-Q-II}.

\begin{remark}\label{rmk-c4}
If \(n>3\), by \eqref{eq-c-4-RHS}, one can deduce that at the point \(Q\) where \(D\varphi = 0\),
\begin{align}\label{eq-c4-Q}
    c_{s,4} = \frac{1}{2n-6} \left( ( \Delta_\Gamma H, \partial_s ) - 2|H|^2 H_s + 2( H_{\xi}, \operatorname{II}_{\xi,\gamma\eta} ) \operatorname{II}_{s,\gamma\eta} \right).
\end{align}
The mean curvature \(H\) appearing in \eqref{eq-c41-Q} and \eqref{eq-c4-Q} is a vector field on \(\Gamma\). Even in the codimension-one case, the term \(( \Delta_\Gamma H, \partial_s )\) differs from \(\Delta_\Gamma \operatorname{H}\), where \(\operatorname{H}\) denotes the scalar mean curvature function.

For example, let \(\fm = n+1\) and let \(\Gamma\) be the graph of \(\varphi\) locally near \(Q\), with \(D\varphi(Q) = 0\). Then at \(Q\),
\begin{align}
    ( \Delta_\Gamma H, \partial_s )&= \left( \Delta_\Gamma \left( \operatorname{H} \cdot \frac{(-\nabla \varphi, 1)}{\sqrt{1+|D\varphi|^2}} \right), \mathbf{e}_s \right) \\
    &= \Delta_\Gamma \operatorname{H} - \operatorname{H} |\operatorname{II}|^2,
\end{align}
where \(\partial_s = \frac{\partial}{\partial y^{n+1}}\). This explains the difference between the expression in \eqref{eq-c-41} and those presented in \cite{Jiang} and \cite{HanJiang2023} when considering the hypersurface case.
\end{remark}

\section{Convergence theorems}\label{sec-Conv}
For simplicity, we set $t=y^n$. Let $\mathbf{w} = (w_1,\cdots, w_{\fm-n})$ be an arbitrary $(\fm-n)$-valued function defined on $B^+_R$, and denote $\mathcal{Q}[\mathbf{w}]$ as specified in \eqref{eq-Q}.

For the solution $\mathbf{u}$ of the minimal surface system, we introduce the auxiliary function
\begin{align}\label{def-v}
    \mathbf{v} = \mathbf{u} - \varphi - \mathbf{c}_2 t^2.
\end{align}
By the argument presented in the proof of Lemma \ref{lem-Form-Comp}, we have the identity
\begin{align}
     \mathcal{Q}\bigl[\varphi + \mathbf{c}_2 t^2\bigr] = t^2 f_1(y),
\end{align}
where $f_1(y)$ is analytic provided that $\varphi$ is analytic.

Combining the minimal surface system \eqref{eq-main} and the derivative formula \eqref{eq-derivative}, we find that the component functions $v_s$ of $\mathbf{v}$ satisfy the following PDE for $s=1,\cdots, \fm-n$:
\begin{align}\label{eq-v}
    g^{ij}[\mathbf{u}]v_{s,ij} - \frac{n}{t}  v_{s,t} = -t^2 f_1(y)+   \bigl(g^{ij}\bigl[\varphi +\mathbf{c}_2 t^2\bigr] -g^{ij}[\mathbf{u}]\bigr)\bigl(\varphi_{s} + c_{s,2}t^2\bigr)_{ij},
\end{align}
where the metric difference term admits the integral representation
\begin{align}
    g^{ij}\bigl[\varphi +\mathbf{c}_2 t^2\bigr] -g^{ij}[\mathbf{u}] &=- \int_0^1 \frac{\partial}{\partial \xi} g^{ij}\bigl[\varphi +\mathbf{c}_2 t^2+ \xi \mathbf{v}\bigr] d\xi\\
    &=\int_0^1  g^{il}\bigl[\varphi +\mathbf{c}_2 t^2+ \xi \mathbf{v}\bigr]\cdot g^{mj}\bigl[\varphi +\mathbf{c}_2 t^2+ \xi \mathbf{v}\bigr] \\
    &\qquad \cdot \left(\mathbf{v}_l \cdot \bigl(\varphi +\mathbf{c}_2 t^2+\zeta \mathbf{v}\bigr)_m +  \mathbf{v}_m \cdot \bigl(\varphi +\mathbf{c}_2 t^2+\zeta \mathbf{v}\bigr)_l    \right) d\zeta .
\end{align}
This metric difference term is smooth with respect to the derivatives of $\mathbf{v}$ (i.e., smooth in $D\mathbf{v}$).

Dividing both sides of \eqref{eq-v} by $g^{nn}=g^{nn}[\mathbf{u}]$, we rewrite the equation as
\begin{align}\label{eq-v-1}
    v_{s, tt} - \frac{n}{t}v_{s, t} = \bigl(g^{nn}\bigr)^{-1}\left(\bigl(g^{nn}v_{s,tt} - g^{ij}v_{s, ij}\bigr) - \frac{n\bigl(g^{nn}-1\bigr)}{t} v_{s,t}-  f_2\right),
\end{align}
where $f_2$ denotes the collection of all terms on the right-hand side of \eqref{eq-v}. After rearrangement, the expression
\begin{align}
g^{nn}v_{s,tt} - g^{ij}v_{s, ij}
\end{align}
contains only mixed and tangential second-order derivatives of $v_s$. Let $(g_{\alpha\beta})$ with $\alpha,\beta = 1,\cdots, n-1$ denote the leading $(n-1)\times (n-1)$ principal submatrix of the metric tensor $(g_{ij})$. We compute the identity
\begin{align}
    g^{nn}-1 = \frac{\det(g_{\alpha\beta} ) - \det (g_{ij})}{\det (g_{ij})} =\frac{(1-g_{nn})\det(g_{\alpha\beta})+\sum_{\beta=1}^{n-1}g_{n \beta} \bigl(g^*\bigr)^{\beta n} }{\det (g_{ij})},
\end{align}
where $g^*$ denotes the cofactor matrix of $g=(g_{ij})$, and each entry $\bigl(g^*\bigr)^{\beta n}$ contains a factor of the form $\mathbf{v}_t + 2\mathbf{c}_2 t$. We deduce that every term in $g^{nn}-1$ involves at least two factors of $\mathbf{v}_t + 2\mathbf{c}_2 t$. Similarly, for each $\beta=1, \cdots, n-1$, every term in $g^{\beta n}$ contains at least one factor of $\mathbf{v}_t + 2\mathbf{c}_2 t$. Based on these observations, we can recast \eqref{eq-v-1} into the form of a Fuchsian system:
\begin{align}\label{eq-Fus}
    t \mathbf{w}_t + A\mathbf{w} = t F\bigl(t, y', \mathbf{w}, D_{y'}\mathbf{w}\bigr),
\end{align}
where $A$ is a constant matrix, and the vector $\mathbf{w}$ is defined by
\begin{align}
    \mathbf{w} = \left(\frac{\mathbf{v}}{t}, \frac{\mathbf{v}_t}{t},\frac{ D_{y'}\mathbf{v}}{t}\right).
\end{align}
Combining the formal computations presented in the proof of Lemma \ref{lem-Form-Comp} with the analytical techniques developed in Kichenassamy \cite{K2}, Kichenassamy-Littman \cite{KL1}, and Kichenassamy-Littman \cite{KL2}, we conclude that for any analytic $(\fm-n)$-valued functions $\varphi$ and $\mathbf{c}_{n+1}$, the Fuchsian system \eqref{eq-Fus} admits a unique series solution that converges locally in a neighborhood of the origin. For further details on the convergence analysis, we refer the reader to Kichenassamy \cite{K1} and Han-Jiang \cite{HanJiang1}. We summarize these findings in the following lemma.

\begin{lemma}\label{lem-Fuschian}
    Let $\varphi$ and $\mathbf{c}_{n+1}$ be arbitrary analytic $(\fm -n)$-valued functions defined on $B'_R$. Then there exists a constant $r\in (0, R)$ such that the series \eqref{eq-expan} converges uniformly and absolutely on $\overline{B^+_r}$, with the limit function being a real-valued solution $\mathbf{u}$ to the minimal surface system \eqref{eq-main} in $B^+_r$.
\end{lemma}

To establish Theorem \ref{thm-Conv-1}, it remains to verify the analyticity of the vector-valued function $\mathbf{c}_{n+1}$. We recast equation \eqref{eq-g-inv3} in the following form for each $s=1,\cdots, \fm -n$:
\begin{align}\label{eq-v-11}
    A_{ij}v_{s,ij} - n \frac{v_{s,t}}{t}+N=0,
\end{align}
where the coefficients $A_{ij}$ and the nonlinear term $N$ depend on the variables $y'$, $t$, and the first derivatives of $\mathbf{v}$, namely
\begin{align}
    A_{ij} = A_{ij}(y',t, D\mathbf v),\quad N = N(y', t, D \mathbf v).
\end{align}

For the solution $\mathbf{v}$, we impose the following conditions: there exists a positive constant $C_0$ such that
\begin{align}
\label{eq-QuasiCondition}
|\mathbf v|\leq C_0t^2,\quad |D\mathbf v|\leq C_0t,\quad |D^2\mathbf v|\leq C_0\quad\text{in }B^+_R.
\end{align}

We further assume that equation \eqref{eq-v-11} is uniformly elliptic: there exists a positive constant $\lambda$ such that for all $\xi\in\mathbb R^n$, $y\in B^+_R$, and $|D\mathbf v|\leq C_0 R$, the following coercivity estimate holds:
\begin{equation}\label{eq-ellipticity} 
\lambda^{-1}|\xi|^2\leq A_{ij}(y, D\mathbf v)\xi_i\xi_j\leq \lambda|\xi|^2.
\end{equation}

An additional structural condition is imposed: there exists a positive constant $c_0$ such that for all $y\in B^+_R$ and $|D\mathbf v|\leq C_0 R$,
\begin{equation}\label{eq-Negative-Condition}2 A_{nn}(y, D\mathbf v)-2n\leq -c_0.\end{equation}

We also require the coefficients $A_{ij}$ and the nonlinear term $N$ to be analytic functions. For notational convenience, we denote the arguments of $A_{ij}$ and $N$ by the pair $(y,z)\in B^+_R\times\mathbb R^{n(\fm-n)}$, where $z$ represents the derivative variables of $\mathbf{v}$. We assume that there exist positive constants $A_0$ and $A$ such that for all nonnegative integers $k,l$, all $y\in B^+_R$, and all $|z| \leq C_0R$, the following derivative bounds hold:
\begin{equation}\label{eq-QuasiLinearAnalyticity-Assum1}
|D^{k+l}_{(y, z)}A_{ij}|+|D^{k+l}_{(y, z)}N|\leq A_0A^{k+l}(k-2)!(l-2)!.
\end{equation}
Here and hereafter, we adopt the convention that $m!=1$ for any integer $m\leq 0$.

For any $k$-valued vector function $\mathbf{w}$, we define its point-wise supremum norm by
\begin{align}
    |\mathbf w(x)| = \max\{|w_1(x)|,\cdots, |w_k(x)|\}.
\end{align}

\begin{theorem}\label{thm-Tan-Ana}
    Suppose that the auxiliary function $\mathbf{v}$ defined by \eqref{def-v} satisfies the PDE system \eqref{eq-v-11} in $B^+_R$ as well as the assumptions \eqref{eq-QuasiCondition}--\eqref{eq-QuasiLinearAnalyticity-Assum1}. 
    Then for any $r\in (0, R/2)$, there exist positive constants $D, B>0$ such that for all $(y',t) \in B^+_r$ and all nonnegative integers $l\geq 0$, the following a priori estimates hold:
    \begin{align}\label{eq-Tan-Ana-1}
        \bigl| D_{y'}^l  \mathbf{v}(y',t)\bigr|& \leq DB^{l-1}(l-1)!\bigl(r - |y'|\bigr)^{-(l-1)^+} t^2,\\\label{eq-Tan-Ana-2}
        \bigl| D D_{y'}^l \mathbf{v}(y',t) \bigr| &\leq DB^{l-1}(l-1)!\bigl(r - |y'|\bigr)^{-(l-1)^+} t,\\\label{eq-Tan-Ana-3}
        \bigl|D^2 D_{y'}^l \mathbf{v}(y',t)\bigr| &\leq DB^{l-1}(l-1)!\bigl(r - |y'|\bigr)^{-(l-1)^+}.
    \end{align}
\end{theorem}

For the base case $l=0$, the estimates reduce precisely to the quasi-homogeneous growth conditions \eqref{eq-QuasiCondition}. Although the proof is rather lengthy, it follows a line of reasoning that is closely analogous to the main approach developed in \cite{HanJiang1}. The key distinction lies in the fact that our argument is tailored to the system of minimal surface equations, whereas the original result in \cite{HanJiang1} addresses a single PDE. For the reader's convenience, we present the complete proof with all details included below.
\begin{proof}

    Because of the interior analyticity of solution $\mathbf u$ to the system of minimal surface equations, we assume that for fixed $r_0\in [r, R)$, \eqref{eq-Tan-Ana-1}-\eqref{eq-Tan-Ana-3} hold in
    \begin{align}\label{eq-Region1}
        \left\{(y', t)\in B^+_{r_0}: t> \frac{r}{2}\right\},
    \end{align}
    for some constants $D, B.$

    For each positive $l$, we set
    \begin{align}
        T_l = \left\{(y', t)\in B^+_r: t<\frac{1}{l}(r-|y'|)\right\}.
    \end{align}
Hence, $T_l$ is a circular cone and shrinks while $l$ increases. In this way, we decompose $B^+_r$
into two parts $T_l$ and $B^+_r \,/\, T_l$.

We prove \eqref{eq-Tan-Ana-1}-\eqref{eq-Tan-Ana-3} by induction. Theorem \ref{thm-BRT1} and Lemma \ref{lem-tan-smooth} imply \eqref{eq-Tan-Ana-1}-\eqref{eq-Tan-Ana-3} for $l=1.$ 
Assume $p\geq 2$ and assume \eqref{eq-Tan-Ana-1}-\eqref{eq-Tan-Ana-3} hold for all $l<p.$

\smallskip
{\it Step 1.} We prove \eqref{eq-Tan-Ana-1} for $l=p$ in $B^+_r$. We consider 
the cases $T_p$ and $B^+_r\setminus T_p$ separately. 

We first take a $y_0=(y_0', t_0)\in T_p$. Set 
\begin{align}
\rho=\frac1p(r-|y_0'|),    
\end{align} 
and 
\begin{align}
    B^+_\rho(\tilde y_0)=\{(y',t):\, |y'-y_0'|<\rho,\, t\in (0,\rho)\},
\end{align}
where $\tilde y_0=(y_0',0)$. The definition of $T_p$ implies $t_0<\rho$. 
Next, we take any $(y',t)\in B^+_\rho(\tilde y_0)$. Then, 
\begin{align}
(r-|y_0'|)-(r-|y'|)=|y'|-|y_0'|\leq |y'-y_0'|<\rho=\frac1p(r-|y'_0|).    
\end{align}
Hence, 
\begin{equation}\label{eq-LinearAnalyticity-Relation}r-|y_0'|<\frac{p}{p-1}(r-|y'|).\end{equation}
A similar argument yields 
\begin{equation}\label{eq-LinearAnalyticity-Relation2}r-|y'|<\frac{p+1}{p}(r-|y_0'|).\end{equation}
With $t<\rho$ in $B^+_\rho(\tilde y_0)$, we have 
\begin{align}t<\rho=\frac{1}{p}(r-|y_0'|)<\frac{1}{p-1}(r-|y'|).\end{align}
This implies 
\begin{equation}\label{eq-LinearAnalyticity-Inclusion}B^+_\rho(\tilde y_0)\subset T_{p-1}.\end{equation}

Consider, for some positive constants $\varepsilon$ and $M$ to be determined, 
\begin{align}{w}(y',t)=M(\varepsilon|y-y_0'|^2+t^2).\end{align}
Setting
\begin{align}
    L w = A_{ij}(y, D\mathbf v(y))w_{ij} - n \frac{w_{t}}{t},
\end{align}
one has
\begin{align}
L w=M\Big(2 A_{nn}-2n + 2\varepsilon\sum_{\beta=1}^{n-1}A_{\beta\beta}\Big).
\label{TS72}
\end{align}
By \eqref{eq-Negative-Condition} and taking $\varepsilon$ small, we have 
\begin{align}Lw\leq -\frac12Mc_0.\end{align}
For simplicity, we assume $c_0\in (0,1]$. 
Next, the definition of $w$ implies  
\begin{align} 
w&\geq  M\varepsilon \rho^2\quad\text{on }\partial B_\rho'(y_0')\times (0,\rho),\\
w&\geq  M\rho^2\quad\text{on }B_\rho'(y_0')\times\{\rho\}.\end{align}
By the induction hypotheses 
\eqref{eq-Tan-Ana-2} for $l=p-1$, we have 
\begin{align}|D_{y'}^p \mathbf v(x)|\leq B_0B^{p-2}(p-2)!t(r-|x'|)^{-p+2}.\end{align}
Note that \eqref{eq-LinearAnalyticity-Relation} implies, for $(y',t)\in B^+_\rho(\tilde y_0)$,  
\begin{align}(r-|y'|)^{-p+2}<\Big(\frac{p-1}{p}\Big)^{-p+2}(r-|y_0'|)^{-p+2}
\leq c_1(r-|y_0'|)^{-p+2},\end{align}
where $c_1$ is a positive constant independent of $p$. Hence by the definition of $\rho$, we get, 
for any $(y',t)\in B^+_\rho(\tilde y_0)$, 
\begin{align} |D_{y'}^p\mathbf v(x)|&\leq c_1B_0B^{p-2}(p-2)!\rho(r-|y_0'|)^{-p+2}\\
&\leq c_1B_0B^{p-2}(p-1)!\rho^2(r-|y_0'|)^{-p+1}.\end{align}
In order to have $w\geq  |D_{y'}^p \mathbf v|$ on $\partial B^+_\rho(\tilde y_0)$, we need to choose, 
by renaming $c_1$,  
\begin{equation}\label{eq-LinearAnalyticity-ChoiceM1}
M\geq  c_1B_0B^{p-2}(p-1)!(r-|y_0'|)^{-p+1}.\end{equation}

By applying $D_{y'}^l$ to \eqref{eq-v-11}, we obtain 
\begin{align}\label{TS91}
L(D_{y'}^lv_s)+N_l=0,
\end{align}
where $N_l$ is given by 
\begin{align}
\label{eq-KeyExpressions2}N_l=
\sum_{m=0}^{l-1}
\left(\begin{matrix}l\\m\end{matrix}\right)\left(D_{y'}^{l-m}A_{ij}\cdot D_{y'}^mv_{s,ij}\right)
+D_{y'}^lN.    
\end{align}
Derivatives of $A_{ij}$ and $N$ also result in derivatives of 
$\mathbf u$.  
We claim that, by taking $B$ sufficiently large depending only on 
$A_0$, $B_0$ and $A$, we have, for any $(y',t)\in B^+_\rho(\tilde y_0)$,
\begin{equation}\label{eq-LinearAnalyticity-Estimate-f}
|N_p(y',t)|
\leq C_1 B_0B^{p-2}(p-1)!(r-|y_0'|)^{-p+1},
\end{equation}
where $C_1$ is a positive constant depending only on $A_0$, $B_0$, and $A$. By renaming 
$C_1$, we may require $C_1\geq  c_1$, for $c_1$ in \eqref{eq-LinearAnalyticity-ChoiceM1}, 
and $C_1\geq  2c_0^{-1}$, for $c_0$ as in \eqref{eq-Negative-Condition}. 
Set 
\begin{equation}\label{eq-LinearAnalyticity-ChoiceM2}
M= C_1 B_0B^{p-2}(p-1)!(r-|y_0'|)^{-p+1}.
\end{equation}
Therefore, we obtain 
\begin{align}
L(\pm D_{y'}^pv_s)&\geq  Lw\quad\text{in }B^+_\rho(\tilde y_0),\\
\pm D_{y'}^pv_s&\leq w\quad\text{on }\partial B^+_\rho(\tilde y_0).   
\end{align}
By the maximum principle, we have
\begin{align}
|D_{y'}^pv_s|\leq w\quad\text{in }B^+_\rho(\tilde y_0).    
\end{align}
By taking $y'=y'_0$, we obtain, for any $(y_0',t)\in B^+_\rho(\tilde y_0)$, 
\begin{align}|D_{y'}^pv_s(y_0',t)|\leq Mt^2.\end{align}
In conclusion, by \eqref{eq-LinearAnalyticity-ChoiceM2},
we obtain, 
for any $(y',t)\in T_p$, 
\begin{align}|D^p_{y'} v_s(y',t)|\leq C_1 B_0B^{p-2}(p-1)!t^2(r-|y'|)^{-p+1}.\end{align}

We now prove \eqref{eq-LinearAnalyticity-Estimate-f}. In view of \eqref{eq-KeyExpressions2} 
with $l=p$, we first estimate $D_{y'}^pN$. 
For any $k=1, \cdots, p$, by taking $l=k-1<p$ in the induction hypothesis 
\eqref{eq-Tan-Ana-2} and \eqref{eq-Tan-Ana-3}, we have 
\begin{align}
\left|D_{y'}^k\frac{\mathbf v}{t}\right|&\leq\frac1t\left|D_{y'}^{k-1}D \mathbf v\right|
\leq B_0B^{(k-2)^+}(k-2)!(r-|y'|)^{-(k-2)^+},\\
|D_{y'}^kD \mathbf v|&\leq|D_{y'}^{k-1}D^2 \mathbf v|\leq B_0B^{(k-2)^+}(k-2)!(r-|y'|)^{-(k-2)^+}.
\end{align}
By Lemma A.1 and Remark A.2 in \cite{HanJiang1}, we obtain 
\begin{equation}\label{eq-EstimateN1}|D^p_{y'}N|\leq \widetilde B_0B^{p-2}(p-2)!(r-|y'|)^{-(p-2)}.
\end{equation}
Next, we estimate terms involving $A_{ij}$ in \eqref{eq-KeyExpressions2}, i.e., 
\begin{equation}\label{eq-estimate-I}I=\sum_{m=0}^{p-1}
\left(\begin{matrix}p\\m\end{matrix}\right)D_{y'}^{p-m}A_{ij} \partial_{ij}D_{y'}^mv_s.\end{equation}
Similar as \eqref{eq-EstimateN1}, we have, for any $k=0, 1, \cdots, p$, 
\begin{align}|D^k_{y'}A_{ij}|\leq \widetilde B_0B^{(k-2)^+}(k-2)!(r-|y'|)^{-(k-2)^+}.\end{align}
In expanding the summation in $I$, we consider $m=0, 1, p-1$ separately. 
By  the induction hypotheses 
\eqref{eq-Tan-Ana-3} for $l<p$, we have 
\begin{align} 
|I|&\leq \widetilde B_0B^{p-2}(p-2)!(r-|y'|)^{-p+2}+\widetilde B_0B_0B^{p-3}p(p-3)!(r-|y'|)^{-p+3}\\
&\qquad +\widetilde B_0B_0B^{p-3}(p-1)!(r-|y'|)^{-p+3}
\sum_{m=2}^{p-2}\frac{p}{m(p-m)(p-m-1)}\\
&\qquad+\widetilde B_0B_0B^{p-2}p(p-2)!(r-|y'|)^{-p+2}.\end{align} 
We note that the last term in the right-hand side above has the order $B^{p-2}(p-1)!$. 
A straightforward calculation yields 
\begin{align}|I|\leq B_1B_0B^{p-2}(p-1)!(r-|y'|)^{-p+1}.\end{align}
Therefore, we obtain \eqref{eq-LinearAnalyticity-Estimate-f}. 

Next, we take $(y',t)\in B^+_r\setminus T_p$. By the induction hypotheses 
\eqref{eq-Tan-Ana-2} for $l=p-1$, we have 
\begin{align}|D^p_{y'} v_s(y',t)|\leq B_0B^{p-2}(p-2)!t (r-|y'|)^{-p+2}.\end{align}
Note $r-|y'|\leq pt$ in $B^+_r\setminus T_p$. Then, 
\begin{align}|D^p_{y'} v_s(y',t)|\leq \frac{p}{p-1}B_0B^{p-2}(p-1)!t^2 (r-|y'|)^{-p+1}.\end{align}

By combining the both cases for points in $T_p$ and $B^+_r\setminus T_p$, 
we obtain, for any $(y', t)\in B^+_r$, 
\begin{equation}\label{eq-Tan-Ana-1-tau}
|D^p_{y'} v_s(y',t)|\leq C_1 B_0B^{p-2}(p-1)! t^2(r-|y'|)^{-p+1}.\end{equation}
This implies \eqref{eq-Tan-Ana-1} for $l=p$, if $B\geq  C_1$. 
The extra factor $B^{-1}$ is for later purposes.

\smallskip

{\it Step 2.} We prove \eqref{eq-Tan-Ana-2} for $l=p$ in $B^+_r$. Again, we consider 
the cases $T_p$ and $B^+_r\setminus T_p$ separately.

Take any $y_0=(y_0', t_0)\in T_p$ and set $\rho=t_0$. Then, $B_\rho(y_0)\subset B^+_r$. 
By a similar argument, \eqref{eq-LinearAnalyticity-Relation}
and \eqref{eq-LinearAnalyticity-Relation2} hold in $B_\rho(y_0)$. Similar to 
\eqref{eq-LinearAnalyticity-Estimate-f}, we have, in $B_\rho(y_0)$,  
\begin{equation}\label{eq-LinearAnalyticity-Estimate-f2}
|N_p|
\leq C_1 B_0B^{p-2}(p-1)!(r-|y_0'|)^{-p+1}.
\end{equation}
We now consider \eqref{TS91} in $B_{3\rho/4}(y_0)$ for $l=p$. 
Note 
\begin{align}|A_{ij}|_{L^\infty(B_{3\rho/4}(y_0))}
+2\rho n\left|t^{-1}\right|_{L^\infty(B_{3\rho/4}(y_0))}
\leq C.\end{align}
We  fix an arbitrary constant $\alpha\in(0,1)$. 
The scaled $C^{1,\alpha}$-estimate implies 
\begin{align} &\rho^{\alpha}[D_{y'}^pv_s]_{C^\alpha(B_{\rho/2}(y_0))}
+\rho|DD_{y'}^pv_s|_{L^\infty(B_{\rho/2}(y_0))}
+\rho^{1+\alpha}[DD_{y'}^pv_s]_{C^\alpha(B_{\rho/2}(y_0))}\\
&\qquad \leq c_2\big(|D_{y'}^pv_s|_{L^\infty(B_{3\rho/4}(y_0))}+\rho^2|N_p|_{L^\infty(B_{3\rho/4}(y_0))}\big).
\end{align}
By \eqref{eq-Tan-Ana-1-tau} and 
\eqref{eq-LinearAnalyticity-Estimate-f2}, we have 
\begin{align}\label{eq-LinearAnalyticity-Estimate-u}\begin{split}
&\rho^{\alpha}[D_{y'}^pv_s]_{C^\alpha(B_{\rho/2}(y_0))}
+\rho|DD_{y'}^pv_s|_{L^\infty(B_{\rho/2}(y_0))}
+\rho^{1+\alpha}[DD_{y'}^pv_s]_{C^\alpha(B_{\rho/2}(y_0))}\\
&\qquad \leq C_2 B_0B^{p-2}(p-1)!\rho^2(r-|y_0'|)^{-p+1}.
\end{split}\end{align}
In particular, we get 
\begin{align}|DD_{y'}^pv_s(y_0)|\leq C_2  B_0B^{p-2}(p-1)!\rho(r-|y_0'|)^{-p+1}.\end{align}

Next, we take $(y',t)\in B^+_r\setminus T_p$. By the induction hypotheses 
\eqref{eq-Tan-Ana-3} for $l=p-1$, we have 
\begin{align}|DD^p_{y'} v_s(y',t)|\leq B_0B^{p-2}(p-2)! (r-|y'|)^{-p+2}.\end{align}
Note $r-|y'|\leq pt$ in $B^+_r\setminus T_p$. Then, 
\begin{align}|DD^p_{y'} v_s(y',t)|\leq \frac{p}{p-1}B_0B^{p-2}(p-1)!t (r-|y'|)^{-p+1}.\end{align}

By combining the both cases for points in $T_p$ and $B^+_r\setminus T_p$, 
we obtain, for any $(y', t)\in B^+_r$, 
\begin{equation}\label{eq-Tan-Ana-2-tau}
|DD^p_{y'} v_s(y',t)|\leq C_2 B_0B^{p-2}(p-1)! t(r-|y'|)^{-p+1}.\end{equation}
This implies \eqref{eq-Tan-Ana-2} for $l=p$, if $B\geq  C_2$. 

{\it Step 3.} We prove \eqref{eq-Tan-Ana-3} in $T_p$ for $l=p$. 

As in Step 2, we take any $y_0=(y_0',t_0)\in T_p$ and set $\rho=t_0$. 
A simple calculation yields
\begin{align}\rho^\alpha[A_{ij}]_{C^\alpha(B_{\rho/2}(y_0))}
+2\rho^{1+\alpha}n[t^{-1}]_{C^\alpha(B_{\rho/2}(y_0))}
\leq c_3.\end{align}
We now consider \eqref{TS91} in $B_{\rho/2}(y_0)$ for $l=p$. 
The scaled $C^{2,\alpha}$-estimate implies 
\begin{align}\rho^2|D^2D_{y'}^p v_s(y_0)|\leq 
c_3\big\{|D_{y'}^pv_s|_{L^\infty(B_{\rho/2}(y_0))}+\rho^2|N_p|_{L^\infty(B_{\rho/2}(y_0))}
+\rho^{2+\alpha}[N_p]_{C^\alpha(B_{\rho/2}(y_0))}\big\}.\end{align}
By \eqref{eq-Tan-Ana-1-tau} and 
\eqref{eq-LinearAnalyticity-Estimate-f2}, we have 
\begin{align}|D^2D_{y'}^pv_s(y_0)|\leq 
C_3 B_0B^{p-2}(p-1)!(r-|y_0'|)^{-p+1}
+c_3\rho^{\alpha}[N_p]_{C^\alpha(B_{\rho/2}(y_0))}.\end{align}
We claim 
\begin{equation}\label{eq-LinearAnalyticity-Estimate-f3}
\rho^{\alpha}[N_p]_{C^\alpha(B_{\rho/2}(y_0))}\leq C_3 B_0B^{p-2}(p-1)!(r-|y_0'|)^{-p+1}.\end{equation}
Hence, 
\begin{align}|D^2D_{y'}^pv_s(y_0)|\leq 
C_3 B_0B^{p-2}(p-1)!(r-|y_0'|)^{-p+1}.\end{align}
By taking $B\geq  C_3$, we obtain, for any $(y',t)\in T_p$, 
\begin{align}|D^2D_{y'}^pv_s(y',t)|\leq B_0B^{p-1}(p-1)!(r-|y_0'|)^{-p+1}.\end{align} 
This is \eqref{eq-Tan-Ana-3} for $l=p$ in $T_p$. 

We now  prove \eqref{eq-LinearAnalyticity-Estimate-f3}
by examining $N_p$ given by \eqref{eq-KeyExpressions2} for $l=p$. We note 
that $N_p$ consists of two parts. The first part is given by a summation 
and the second part 
by $D_{y'}^pN$. For $D_{y'}^pN$, we have 
\begin{equation}\label{eq-LinearAnalyticity-Estimate-N1}
\rho^{\alpha}[D_{y'}^pN]_{C^\alpha(B_{\rho/2}(y_0))}
\leq \widetilde B_0B^{p-2}(p-2)!(r-|y_0'|)^{-(p-2)}.\end{equation}
The proof is similar to that of \eqref{eq-EstimateN1}. We point out that 
Lemma A.1 in \cite{HanJiang1} still holds if the $L^\infty$-norms are 
replaced by $C^\alpha$-norms and the needed estimates of the 
$C^\alpha$ semi-norms of $DD_{y'}^l\mathbf v$ and 
$D_{y'}^l\mathbf v/t$ are provided by \eqref{eq-LinearAnalyticity-Estimate-u}, for $l\leq p$. 
Next, we examine the summation part in $N_p$ and discuss $I$ in 
\eqref{eq-estimate-I} for an illustration. Similar to 
\eqref{eq-LinearAnalyticity-Estimate-N1}, we have, for any $l\leq p$,  
\begin{align}\rho^{\alpha}[D_{y'}^lA_{ij}]_{C^\alpha(B_{\rho/2}(y_0))}
\leq \widetilde B_0B^{(l-2)^+}(l-2)!(r-|y_0'|)^{-(l-2)^+}.\end{align}
We note 
that $I$ is a linear combination of $D^2D_{y'}^{m}\mathbf \mathbf v$, for $m\leq p-1$, 
which can be written as $DD_{y'}^m \mathbf v$ for $m\leq p$ and 
$\partial_t^2D_{y'}^m \mathbf v$ for $m\leq p-1$. We estimate these two groups separately. 
To do this, we first have, for any $l\leq p$, 
\begin{align}\label{eq-LinearAnalyticity-Estimate-u0}\begin{split}
&|D_{y'}^l\mathbf v|_{L^\infty(B_{\rho/2}(y_0))}
+\rho^{\alpha}[D_{y'}^l \mathbf v]_{C^\alpha(B_{\rho/2}(y_0))}\\
&\qquad+\rho|DD_{y'}^l \mathbf v|_{L^\infty(B_{\rho/2}(y_0))}
+\rho^{1+\alpha}[DD_{y'}^l \mathbf v]_{C^\alpha(B_{\rho/2}(y_0))}\\
&\qquad \leq C_2 B_0B^{(l-2)^+}(l-1)!\rho^2(r-|y_0'|)^{-(l-1)^+}.
\end{split}\end{align}
We note that \eqref{eq-LinearAnalyticity-Estimate-u0} is implied by 
\eqref{eq-Tan-Ana-1-tau} and 
\eqref{eq-LinearAnalyticity-Estimate-u} for $l=p$. The proof in Step 2 actually shows 
that \eqref{eq-LinearAnalyticity-Estimate-u0} holds for all $l\leq p$. 
Next, we prove, for $l\leq p-1$, 
\begin{align}\label{eq-LinearAnalyticity-Estimate-u0a}\begin{split}
&\rho|\partial_t^2D_{y'}^l\mathbf v|_{L^\infty(B_{\rho/2}(y_0))}
+\rho^{1+\alpha}[\partial_t^2D_{y'}^l \mathbf v]_{C^\alpha(B_{\rho/2}(y_0))}\\
&\qquad \leq C_3B_0B^{(l-1)^+}l!\rho^2(r-|y_0'|)^{-l}.
\end{split}\end{align}
To prove \eqref{eq-LinearAnalyticity-Estimate-u0a}, we first have, by 
\eqref{eq-v-11},
\begin{equation}v_{s,tt}=-\frac{N}{A_{nn}}
-\sum_{0\leq i+j\leq 2n-1}\frac{A_{ij}}{A_{nn}}v_{s,ij}+\frac1t\frac{n}{A_{nn}}v_{s,t}
.\end{equation}
Then, for $l\leq p-1$, 
\begin{equation}\label{eq-equation-l}D_{y'}^{l}v_{s,tt}=-D_{y'}^{l}\Big(\frac{N}{A_{nn}}
+\sum_{0\leq i+j\leq 2n-1}\frac{A_{ij}}{A_{nn}}v_{s,ij}-\frac1t\frac{n}{A_{nn}}v_{s,t}
\Big).\end{equation}
We analyze the summation involving $A_{ij}$. 
For each pair $i$ and $j$ with $i+j<2n$, $ v_{s,ij}$ is a part of $DD_{y'}v_s$. 
Hence, for $l\leq p-1$, $D_{y'}^l(A_{nn}^{-1}A_{ij}v_{s,ij})$ is a linear combination of 
$DD_{y'}^m v_s$, for $m=1, \cdots, p$. The $C^\alpha$-norms of these derivatives of $v_s$ 
are already estimated by 
\eqref{eq-LinearAnalyticity-Estimate-u0}. We can analyze other terms similarly. Hence, we have 
\eqref{eq-LinearAnalyticity-Estimate-u0a}. As a consequence, we get 
\begin{equation}
\rho^{\alpha}[I]_{C^\alpha(B_{\rho/2}(y_0))}\leq C_3 B_0B^{p-2}(p-1)!(r-|y_0'|)^{-p+1}.\end{equation}
We can analyze other terms in $N_p$ similarly. Therefore, we obtain
\eqref{eq-LinearAnalyticity-Estimate-f3} and finish the proof of the claim.

\smallskip

{\it Step 4.} We prove 
\eqref{eq-Tan-Ana-3} in $B^+_r\setminus T_p$ for $l=p$. We will fix $r$ in this step. 

Take any $y_0=(y_0', t_0)\in B^+_r\setminus T_p$, with $t_0\leq r/2$. Then, $t_0\geq  (r-|y_0'|)/p$. 
Set 
\begin{align}\rho=\frac{1}{2p}(r-|y_0'|).\end{align}
Then, $t_0\geq  2\rho$. Hence, for any $(y',t)\in B_\rho(y_0)$, $t\geq  t_0-\rho\geq  \rho$. 
We now consider \eqref{TS91} in $B_{\rho}(y_0)$ for $l=p+1$. 
Note 
\begin{align}|A_{ij}|_{L^\infty(B_{\rho}(y_0))}
+2n\rho\left|t^{-1}\right|_{L^\infty(B_{\rho}(y_0))}
\leq c_4.\end{align}
We  fix an arbitrary constant $\alpha\in(0,1)$. 
The scaled $C^{1,\alpha}$-estimate implies 
\begin{align}\rho|DD_{y'}^{p+1}v_s(y_0)|\leq c_4\big\{|D_{y'}^{p+1}v_s|_{L^\infty(B_\rho(y_0))}
+\rho^2|N_{p+1}|_{L^\infty(B_\rho(y_0))}\big\}.\end{align}
By the induction hypotheses \eqref{eq-Tan-Ana-3} for $l=p-1$, we have 
\begin{align}|D_{y'}^{p+1}v_s(y)|\leq B_0B^{p-2}(p-2)!(r-|y'|)^{-p+2}.\end{align}
By a similar argument, \eqref{eq-LinearAnalyticity-Relation}
and \eqref{eq-LinearAnalyticity-Relation2} hold in $B_\rho(y_0)$. 
Hence, for any $y=(y',t)\in B_\rho(y_0)$, 
\begin{align}|D_{y'}^{p+1}v_s(y)|&\leq c_4B_0B^{p-2}(p-2)!(r-|y'_0|)^{-p+2}\\
&\leq c_4B_0B^{p-1}(p-1)!\rho(r-|y'_0|)^{-p+1}.\end{align}
Next, we consider \eqref{eq-KeyExpressions2} for $l=p+1$. 
In the expression of $N_{p+1}$, we single out the term $D^2D_{y'}^{p}\mathbf v$. We note that 
$D_{y'}^l\mathbf v$, $DD_{y'}^l\mathbf v$, $D^2D_{y'}^l\mathbf v$ can be estimated by the induction hypothesis, for $l<p$, 
and that $D_{y'}^p\mathbf v$, $DD_{y'}^p\mathbf v$ can be estimated by Step 1 and Step 2, respectively. 
Hence, a similar argument as in Step 1 yields 
\begin{align}|N_{p+1}|_{L^\infty(B_\rho(y_0))}\leq (p+1)A_0A|D^2D_{y'}^{p}\mathbf v|_{L^\infty(B_\rho(y_0))}
+C_1 B_0B^{p-2}(p-1)!(r-|y'_0|)^{-p+1}.\end{align}
By a simple substitution, we have 
\begin{align}|DD_{y'}^{p+1}v_s(y_0)|\leq (p+1)A_0A\rho|D^2D_{y'}^{p}\mathbf v|_{L^\infty(B_\rho(y_0))}
+C_1 B_0B^{p-2}(p-1)!(r-|y'_0|)^{-p+1}.\end{align}
Combining with \eqref{eq-equation-l} for $l=p$, we get 
\begin{align}|D^2D_{y'}^{p}v_s(y_0)|\leq (p+1)A_0A\rho|D^2D_{y'}^{p}\mathbf v|_{L^\infty(B_\rho(y_0))}
+C_4 B_0B^{p-2}(p-1)!(r-|y'_0|)^{-p+1}.\end{align}
We now fix a constant $\varepsilon\in (0,1)$. 
By the definition of $\rho$, we can choose $r$ sufficiently small such that
\begin{align}|D^2D_{y'}^{p}\mathbf v(y_0)|\leq \varepsilon|D^2D_{y'}^{p}\mathbf v|_{L^\infty(B_\rho(y_0))}
+C_4 B_0B^{p-2}(p-1)!(r-|y'_0|)^{-p+1}.\end{align}
Next, for any $\eta\in (0,r)$, we define 
\begin{align}h(\eta)=\sup\{|D^2D_{y'}^{p}\mathbf v|:\, y\in B^+_r\setminus T_p,\,  |y'|\leq \eta\}.\end{align}
At points in $B_{\rho}(y_0)\cap T_p$, $D^2 D^{p}_{y'}\mathbf v$ is already bounded in Step 3. 
Hence, we have, for any $\eta\in (0,r)$,  
\begin{align}
h(\eta) \leq \varepsilon h\big(\eta+p^{-1}(r-\eta)\big)+C_4 B_0B^{p-2}(p-1)!(r-\eta)^{-p+1}.
\end{align}
By applying Lemma \ref{lemma-LinearAnalyticity-Iteration} below to the function $h$, we obtain, 
for any $\eta\in (0,r)$, 
\begin{align}h(\eta)\leq CC_4 B_0B^{p-2}(p-1)!(r-\eta)^{-p+1}.\end{align}
We now choose $B\geq  CC_4$. For each $(y',t)\in B^+_r\setminus T_p$, we take 
$\eta=|y'|$ and then obtain 
\begin{align}|D^2D_{y'}^{p}\mathbf v(y',t)|\leq B_0B^{p-1}(p-1)!(r-|y'|)^{-p+1}.\end{align}
This ends the proof of \eqref{eq-Tan-Ana-3} in $B^+_r\setminus T_p$ for $l=p$. 

In summary, we take $B\geq  \max\{C_1, C_2, C_3, CC_4\}$. 

While we select $r$ to be sufficiently small in Step 4, for any $r_0 \in (r, R/2)$, the region $B^+_{r_0}$ can generally be covered by the domain in \eqref{eq-Region1} and translates of $B^+_r$ (with distinct centers in $B_{r_0}'$) such that \eqref{eq-Tan-Ana-1}-\eqref{eq-Tan-Ana-3} hold in $B^+_{r_0}$.
\end{proof}

We need the following lemma to finish the proof of Theorem \ref{thm-Tan-Ana}. See Lemma 2 in \cite{Friedman1958}. 

\begin{lemma}\label{lemma-LinearAnalyticity-Iteration}
Let $p$ be a positive integer, 
$\varepsilon\in (0,1)$ and $M>0$ be constants, 
and $h(t)$ be a positive monotone increasing function defined in the interval $[0,r]$. 
Assume, for any $\eta\in (0,r)$, 
\begin{align}
h(\eta) \leq \varepsilon h\big(\eta+p^{-1}(r-\eta)\big)+M(r-\eta)^{-p}.
\end{align}
Then, for any $\eta\in (0,r)$, 
\begin{align}h(\eta)\leq CM(r-\eta)^{-p},\end{align} 
where $C$ is a positive constant depending only on $\varepsilon$, independent of $p$.
\end{lemma}

The analyticity of  $\mathbf{c}_{n+1,0}$ follows from Theorem \ref{thm-Tan-Ana}. For a comprehensive discussion of the underlying reasoning, we refer the reader to Section 3 of \cite{HanJiang1}. In fact, $\mathbf{c}_{n+1,0}$ can be represented as a linear combination of the coefficients arising from the solution to a specific ordinary differential equation (ODE), where each coefficient is essentially an integral of the auxiliary function $\mathbf{v}$.

We then invoke Lemma \ref{lem-Fuschian} together with a unique continuation result established in \cite{HanJiang1} to deduce Theorem \ref{thm-Conv-1}. For additional background on the unique continuation theorem, we also refer to \cite{Mazzeo}.

The following theorem is a direct corollary of the expansion formula \eqref{eq-u-exp-0} combined with Theorem 6.1 in \cite{HanJiang1}.
\begin{theorem}[Convergence Theorem Under Smooth Assumption]\label{thm-conv-2}
There exists some constant $\rho_\Gamma>0$ such that
if Assumption \ref{assmp-local} holds and 
\begin{align}
\Gamma \cap \overline {G_R} \in C^{\infty},
\end{align}
 for some $R\in (0, \rho_\Gamma]$,
then for any $r\in (0, R),$   $\operatorname{supp}(T) \cap G_r$ is the graph of an analytic $(\fm-n)$-valued function $\mathbf u$ defined on $B^+_r$ in the Euclidean metric, and the following holds:
\begin{enumerate}
\item there are smooth $(\fm-n)$-valued functions $\mathbf w_0, \mathbf w_1,\cdots$, independent of the choice of $r,$ such that
\begin{align}
    \mathbf u=\mathbf  w_0 + \sum^{\infty}_{j=1} \mathbf w_j\cdot (-(y^n)^n \log (y^n))^{j} ,
\end{align}
absolutely and uniformly in $\bar{G}_r$, 
and $\mathbf w_j(x', 0)=0$ for $j\geq 1$;

\item
for any $i\in \{1,\cdots, n\}$ and any $r\in (0, R)$,
 $\mathbf u$ can be expressed as a smooth function of
 \begin{align}
     y, S = - (y^n)^i \log (y^n),
 \end{align}
 such that $\mathbf u$ is analytic in $S$ in $\overline{G_r} \times [0, S_0]$ for some $S_0$ depending on $r$. 
If in addition $i\neq n,$
then for any $k\in \mathbb N,$
\begin{align}
    \partial^k_{y} \mathbf u(y, S) = \partial^k_{y}\mathbf w_0 +\sum^{\infty}_{j=1}(-1)^j \partial^k_{y}\mathbf w_j \cdot S^{j},
\end{align}
absolutely and uniformly in $\overline{G_r} \times [0, S_0]$.
\end{enumerate}
\end{theorem}


\begin{thebibliography}{99}

\bibitem{Almgren} J.F.J. Almgren, {\it Almgren’s Big Regularity Paper, World Scientific Monograph Series in Mathematics}, vol. 1, World Scentific Publishing Co. Inc., River Edge, NJ, 2000.


\bibitem{Anderson1982}
M. Anderson, {\it Complete minimal varieties in hyperbolic space},
Invent. Math., 69(1982), 477-494.

\bibitem{Anderson1983}
M. Anderson, {\it Complete minimal hypersurfaces in hyperbolic {$n$}-manifolds},
Comment. Math. Helv., 58(1983), 264-290.


\bibitem{Federer}
H. Federer, {\it Geometric Measure Theory}, Springer, Berlin-Heidelberg-New York, 1969.

\bibitem{Friedman1958} A. Friedman, \emph{On the regularity of the solutions of nonlinear elliptic and 
parabolic systems of partial differential equations}, J. Math. Mech., 7(1958), 43-59.

\bibitem{FHJ1}
X. Fu, H.-J. Hein, X. Jiang, {\it
Asymptotics of K\"ahler-Einstein metrics on complex hyperbolic cusps}, Calc. Val. P.D.E. (2024). https://doi.org/10.1007/s00526-023-02613-4.

\bibitem{FHJ2}
X. Fu, H.-J. Hein, X. Jiang, {\it
A continuous cusp closing process for negative K\"ahler-Einstein metrics}, 
Geom. Funct. Anal. (2025). https://doi.org/10.1007/s00039-025-00708-y.

\bibitem{HanJiang2023} Q. Han, X. Jiang, {\it Boundary regularity of minimal graphs in the hyperbolic space}, Journal für die reine und angewandte Mathematik (Crelles Journal), vol. 2023, no. 801, 2023, pp. 239-272. https://doi.org/10.1515/crelle-2023-0040

\bibitem{HanJiang2024} Q. Han, X. Jiang, {\it Asymptotics and Convergence for the Complex Monge-Ampère Equation}, Ann. PDE, 10, 8 (2024). https://doi.org/10.1007/s40818-024-00171-2

\bibitem{HanJiang1} Q. Han, X. Jiang, {\it The convergence of boundary expansions and the analyticity of minimal surfaces in the hyperbolic space}, to appear at Adv. in Math., ArXiv preprint: https://arxiv.org/abs/1801.08348

\bibitem{Han16CalVar} Q. Han, W. Shen, Y. Wang, {\it Optimal regularity of minimal graphs in the hyperbolic 
space}, Cal. Var. \& P. D. E., 55(2016), no. 1, Art. 3, 19pp. 

\bibitem{HanWang} Q. Han, Z. Wang, {Solutions of the minimal surface equation and of the Monge–Ampère equation near infinity}, Journal für die reine und angewandte Mathematik (Crelles Journal), 2024. https://doi.org/10.1515/crelle-2024-0091


\bibitem{HanXie} Q. Han, J. Xie,
{\it Uniformly degenerate elliptic equations with varying characteristic exponents}, 
ArXiv preprint: https://arxiv.org/abs/2411.17016


\bibitem{Hardt&Lin1987} R. Hardt, F.-H. Lin, \emph{Regularity at infinity 
for area-minimizing hypersurfaces in hyperbolic space}, Invent. Math., 88(1987), 217-224. 

\bibitem{Jiang}
X. Jiang, {\it Boundary Expansions for Minimal Graphs in the Hyperbolic Space}, University of Notre Dame, Thesis, https://doi.org/10.7274/5712m615m59.

\bibitem{JiangShi}X. Jiang, Y. Shi, {\it Asymptotic expansions of complete Kähler-Einstein metrics with finite volume on quasi-projective manifolds}, Sci. China Math., 65, 1953–1974 (2022). https://doi.org/10.1007/s11425-021-1903-7.

\bibitem{JSZ}
X. Jiang, Y. Sire, R. Zhang, {\it
The singular sets of degenerate and nonlocal elliptic equations on Poincaré-Einstein manifolds},
J. Eur. Math. Soc. (2025), published online first. https://doi.org/10.4171/JEMS/1740.

\bibitem{JiangXiao}X. Jiang, L. Xiao, {\it Optimal regularity of constant curvature graphs in Hyperbolic space},
Calc. Var. P.D.E., 58:133 (2019).

\bibitem{K1} 
S. Kichenassamy, {\it On a conjecture of Fefferman and Graham},
Adv. Math., 184(2004), 268-288.

\bibitem{K2} 
S. Kichenassamy, {\it Fuchsian Reduction: Applications to Geometry},
Cosmology and Mathematical Physics, Birkh\"auser, Boston (2007).

\bibitem{KL1} S. Kichenassamy, W. Littman,
{\it Blow-up surfaces for nonlinear wave equations, Part I},
Comm. P. D. E., 18(1993), 431-452. 


\bibitem{KL2} S. Kichenassamy, W. Littman, 
{\it Blow-up surfaces for nonlinear wave equations, Part II},
Comm. P. D. E., 18(1993), 1869-1899.

\bibitem{LO}
H.B.Lawson, R. Osserman. {\it Non-existence, non-uniqueness and irregularity of solutions to the minimal surface system}, Acta Math., 139, 1–17 (1977). 

\bibitem{Lin1989CPAM} F.-H. Lin, \emph{Asymptotic behavior of area-minimizing 
currents in hyperbolic space}, Comm. Pure Appl. Math., 42(1989), 229-242. 

\bibitem{Lin1989Invent} F.-H. Lin,
{\it On the Dirichlet problem for minimal graphs in hyperbolic space},
Invent. Math., 96(1989), 593-612.

\bibitem{Lin2012Invent} F.-H. Lin,
{\it Erratum: On the Dirichlet problem for minimal graphs in hyperbolic space},
Invent. Math., 187(2012), 755-757.



\bibitem{Mazzeo} R. Mazzeo,
\emph{Unique continuation at infinity and embedded eigenvalues for asymptotically hyperbolic manifolds}, 
Amer. J. Math., 113(1991), 25-45.

\bibitem{Tonegawa1996MathZ} Y. Tonegawa, 
\emph{Existence and regularity of constant mean curvature hypersurfaces in hyperbolic space}, 
Math. Z., 221(1996), 591-615. 
\end{thebibliography}
\end{document}